\documentclass[11pt]{amsart}

\usepackage{amssymb,amsmath,amscd,graphicx,
latexsym,amsthm}
\usepackage{amssymb,latexsym,amsmath,amscd,graphicx}
\usepackage[mathscr]{eucal}
  \usepackage[all]{xy}
        \def\Q{{\mathcal Q}}

\def\OO{{\mathcal O}}

\def\F{\mathcal{F}}

\def\G{\mathcal{G}}

\def\I{\mathcal{I}}

\def\Pic0{{\rm Pic}^0}

\def\l{{\underline l}}
\def\n{{\underline n}}

\def\cR{{\mathcal R}}

\theoremstyle{plain}

\newtheorem{theorem}{Theorem}[section]
\newtheorem{theoremalpha}{Theorem}
\newtheorem{corollaryalpha}[theoremalpha]{Corollary}
\newtheorem{propositionalpha}[theoremalpha]{Proposition}
\newtheorem{proposition/example}[theorem]{Proposition/Example}
\newtheorem{proposition}[theorem]{Proposition}

\newtheorem{corollary}[theorem]{Corollary}
\newtheorem{lemma}[theorem]{Lemma}

\theoremstyle{definition}
\newtheorem{definition}[theorem]{Definition}

\newtheorem{remark}[theorem]{Remark}
\newtheorem{example}[theorem]{Example}

\newtheorem{conjecture/question}[theorem]{Conjecture/Question}

\newtheorem{remark/definition}[theorem]{Remark/Definition}
\newtheorem{notation/assumptions}[theorem]{Assumptions/Notation}

\numberwithin{equation}{section}

 \theoremstyle{remark}

\begin{document}

\title{Cohomological rank  functions on abelian varieties}
\author{Zhi Jiang}
  \address{Shanghai Center for Mathematical Sciences, 22F East Guanghua Tower, Fudan University, No. 220 Handan Road, Shanghai, China  }
 \email{{\tt zhijiang@fudan.edu.cn}}
 \thanks{}

\author{Giuseppe Pareschi}
  \address{Dipartimento di Matematica, Universit\`a di Roma, Tor Vergata, V.le della Ricerca Scientifica, I-00133 Roma, Italy}
 \email{{\tt pareschi@mat.uniroma2.it}}
 \thanks{The first author was partially supported by the grant ``Recruitment Program of Global Experts". The second author was partially supported by Tor Vergata funds E82F16000470005 and Italian MIUR-PRIN funds ``Geometry of Algebraic Varieties".
 He also thanks the Shanghai Center for Mathematical Sciences and Fudan University for support and excellent hospitality }

\maketitle

\setlength{\parskip}{.1 in}

\begin{abstract}Generalizing the \emph{continuous rank function} of Barja-Pardini-Stoppino,  in this paper we consider \emph{cohomological rank functions} of $\mathbb Q$-twisted (complexes of) coherent sheaves on abelian varieties. They satisfy a  natural transformation formula with respect to the Fourier-Mukai-Poincar\'e transform,
which has several consequences. In many concrete geometric contexts these functions provide useful invariants. We illustrate this with two different applications, the first one to GV-subschemes and the second one to multiplication maps of global sections of ample line bundles on abelian varieties.
\end{abstract}

\section*{Introduction}

In their paper \cite{bps} M.A. Barja, R. Pardini and L. Stoppino introduce and study the \emph{continuous rank function} associated to a line bundle $M$ on a variety $X$ equipped with a morphism  $X\buildrel f\over\rightarrow A$  to a polarized abelian variety.
Motivated by their work,   we consider more generally \emph{cohomological rank functions} -- defined in a similar way -- of a bounded complex $\F$  of coherent sheaves on a polarized abelian variety $(A,\l)$ defined over an algebraically closed  field of characteristic zero.  As it turns out, these functions often encode interesting
 geometric information.  The purpose of this paper is to establish some general structure results  about them and show some examples of application.

Let $L$ be an ample line bundle on an abelian variety $A$, let $\l=c_1(L) $ and let $\varphi_\l:A\rightarrow \widehat A$ be the corresponding isogeny. The cohomological rank functions of $\F\in \mathrm{D}^b(A)$ with respect to the polarization $\l$ are initially defined (see Definition \ref{def} below) as certain continuous rational-valued functions
\[h^i_{\F,\l}:\mathbb Q\rightarrow\mathbb Q^{\ge 0}. \>\footnote{In the present context the above-mentioned  continuous rank function of Barja-Pardini-Stoppino  is recovered as $h^0_{f_*M,\, \l}$ (see the notation above).} \]
The definition of these functions is peculiar to abelian varieties (and more generally to irregular varieties), as it uses the isogenies $\mu_b:A\rightarrow A$, \ $z\mapsto bz$. For  $x\in\mathbb Z$,  \ $h^i_{\F,\l}(x):=h^i_\F(x\l)$ coincides with the generic value of $h^i(A,\F\otimes L^x)$, for $L$ varying among all line bundles representing $\l$. This is extended to all $x\in\mathbb Q$ using the  isogenies $\mu_b$. In fact the rational numbers $h^i_\F(x\l)$  can be interpreted as generic cohomology ranks of the $\mathbb Q$-twisted coherent sheaf (or, more generally, $\mathbb Q$-twisted complex of coherent sheaves) $\F\langle x\l\rangle$ (in the sense of Lazarsfeld \cite[\S6.2A]{laz2}).

The above functions are closely related to the Fourier-Mukai transform $\Phi_{\mathcal P}:\mathrm{D}^b(A)\rightarrow \mathrm{D}^b(\widehat A)$ associated to the Poincar\'e line bundle and our first point consists in exploiting systematically such relation. We prove the following transformation formula (Proposition \ref{inversion-a} below)
\begin{align}\label{align}
h^i_{\F}(x\l)\> = &
\> \frac{(-x)^g}{\chi(\l)}h^i_{\varphi_\l^*\Phi_{\mathcal P}(\F)}(-\frac{1}{x}\l)&\quad\hbox{for $x\in \mathbb Q^-$}\\
h^i_{\F}(x\l) \> = &
\> \frac{x^g}{\chi(\l)}h^{g-i}_{\varphi_\l^*\Phi_{\mathcal P^\vee}(\F^\vee)}(\frac{1}{x}\l)&\quad\hbox{for $x\in \mathbb Q^+$}
\end{align}

This has several consequences, summarized in the following theorem. The proof and discussion of the various items are found in Sections 2,3 and 4.
\begin{theoremalpha}\label{summary} Let $\F\in \mathrm{D}^b(A)$ and $i\in\mathbb Z$. Let $g=\dim A$.\\
\emph{(1) (Corollaries \ref{inversion}, \ref{inversion-Q}, \ref{continuity}.) } For each $x_0\in\mathbb Q$ there are $\epsilon^-,\epsilon^+>0$ and two \emph{(explicit, see below)} polynomials $P_{i,\, \F,\,x_0}^+,P_{i,\, \F,\,x_0}^-\in\mathbb Q[x]$ of degree $\le g$ such that $P_{i,\, \F,\,x_0}^+(x_0)=P_{i,\, \F,\,x_0}^-(x_0)$ and
\begin{eqnarray*}h^i_{\F}(x\l)=&P_{i,\,\F,\,x_0}^-(x)&\quad\hbox{for $x\in (x_0-\epsilon^-,x_0]\cap\mathbb Q$} \\
 h^i_{\F}(x\l)=&P^+_{i,\, \F,\,x_0}(x)&\quad\hbox{for $x\in [x_0,x_0+\epsilon^+)\cap\mathbb Q$}
 \end{eqnarray*}

\noindent \emph{(2) (Proposition \ref{derivatives})} Let $k<g$ and $x_0\in\mathbb Q$. If the function $h^i_{\F,\l}$ is strictly of class ${\mathcal C}^k$ at $x_0$ then the jump locus $J^{i+}(\F\langle x_0\rangle)$ \emph{(see \S4 for the definition)} has codimension $\le  k+1$.

 \noindent \emph{(3) (Theorem \ref{c0})} The function $h^i_{\F,\l}$ extends to a continuous function $h^i_{\F,\l}:\mathbb R\rightarrow \mathbb R^{\ge 0}$.
 \footnote{ This theorem provides partial answers to some questions raised,   in the specific case of the  above mentioned continuous rank functions $h^0_{f_*M,\, \l}$, in \cite{bps}, e.g. Question 8.11. We also point out that for such functions item (3) of the present Theorem, as well as some additional properties, were already proved  in \emph{loc. cit.} via different methods.}
 \end{theoremalpha}
    It follows from (1)  that for $x_0\in\mathbb Q$ the function $h^i_{\F,\l}$ is smooth at $x_0$ if and only if the two polynomials  $P^-_{i,\,\F,\, x_0}$ and $P^+_{i,\,\F,\, x_0}$ coincide. If this is not the case $x_0$ is called a \emph{critical point}.

    It turns out (Corollary \ref{inversion}) that for $x_0\in\mathbb Z$ the two polynomials $P^-_{i,\,\F\,\,x_0}(x)$ and $P^+_{i,\,\F,\,x_0}(x)$ are obtained   from  the Hilbert polynomials (with respect to the polarization $\l$) of the two coherent sheaves $\G^{i,+}_{x_0}:=\varphi_\l^*R^i\Phi_{\mathcal P}(\F\otimes L^{x_0})$  and $\G^{i,-}_{x_0}:=\varphi_\l^*R^{g-i}\Phi_{\mathcal P^\vee}(\F^\vee\otimes L^{-x_0})$ in the following way:
  \begin{align}\label{align2}P^-_{i,\,x_0,\F}(x)\> = &
\> \frac{(-x)^g}{\chi(\l)}\chi_{\G^{i,+}_{x_0}}(-\frac{1}{x}\l)\\
P^+_{i,\,x_0,\F}(x)\>=&\> \frac{x^g}{\chi(\l)}\chi_{\G^{i,-}_{x_0}}(\frac{1}{x}\l)\>.
  \end{align}
  For  non-integer $x_0\in \mathbb Q$ the two polynomials  $P^-_{i,\,\F\,\,x_0}(x)$ and $P^+_{i,\,\F,\,x_0}(x)$  have a similar description after reducing to the integer case (Corollary \ref{inversion-Q}).
   Thus item (2) of Theorem \ref{summary} tells that, for $x_0\in\mathbb Q$,  the first $k$ coefficients
    of the
   polynomials $P^-_{i,\,x_0,\F}(x)$ and $P^+_{i,\,x_0,\F}(x)$
  coincide  as soon as  the rank function \ $\Pic0 A\rightarrow \mathbb Z^{\ge 0}$ defined by  $\alpha\mapsto h^i(\F\otimes L^{x_0}\otimes P_\alpha)$  has jump locus of codimension $\ge  k+1$. In this last formulation  we are implicitly assuming that $x_0$ is integer but for rational $x_0$ the situation is completely similar. However there might be irrational critical points (see e.g. Example \ref{irrational}), and at present we lack any similar interpretation for them.

In Section 5 we relate cohomological rank functions with the notions of GV, M-regular and IT(0)-sheaves, which are extended here to the $\mathbb Q$-twisted setting. We provide  formulations of Hacon's results (\cite{hac}), and some related ones, which are  simpler and more convenient even for usual  sheaves. Finally, in Section 6 we point out some integral properties of cohomological rank functions.

 It seems that  the critical points of the function
  and the  polynomials $P^-_{i,\> \F,\, x_0}$ and $P^+_{i,\F,x_0}$ are interesting and sometimes novel invariants in many concrete geometric situations. We exemplify this in the following two applications.

 \noindent{\bf Application to GV-subschemes. } Our first example concerns GV-subschemes of principally polarized abelian varieties (here will assume that the ground field is $\mathbb C$). This notion (we refer to Section 7 below for the definition and basic properties) was introduced in \cite{minimal} in the attempt of providing a Fourier-Mukai approach to the minimal class conjecture (\cite{debarre}), predicting that the only effective algebraic cycles representing the minimal classes $\frac{\underline\theta^{g-d}}{(g-d)!}\in H^{2(g-d)}(A,\mathbb Z)$ are (translates of) the  subvarieties $\pm W_d(C)$ of Jacobians $J(C)$, and  $\pm F$, the Fano surface in the intermediate  Jacobian of a cubic threefold.

  It is known that the subvarieties $W_d(C)$ of Jacobians, as well as the Fano surface
   (\cite{ho1}) are GV-subschemes  and that, on the other hand, geometrically non-degenerate GV-subschemes have minimal classes (\cite{minimal}). Therefore it was conjectured in \emph{loc. cit.} that geometrically non-degenerate GV-subschemes are either (translates of) $\pm W_d(C)$ or $\pm F$ as above.  Denoting $g$ the dimension of the p.p.a.v. and $d$ the dimension of the subscheme, this is known only in a few cases: (i) for $d=1$ and $d=g-2$ (\emph{loc. cit.});
    (ii)  for $g=5$,  settled in the recent work \cite{cps},  (iii)  for Jacobians and   intermediate Jacobians of generic cubic threefolds,  as consequences of the main results of respectively \cite{debarre} and \cite{ho2}. In the recent work \cite{s} it is proved that geometrically non-degenerate GV-subschemes are reduced and irreducible and that  the geometric genus of their desingularizations is the expected one, namely $g\choose d$.

  As an application of cohomological rank functions we prove that the Hilbert polynomial as well as all  $h^i(\OO_X)$'s are the expected ones:
 \begin{theoremalpha}\label{introGV} Let $X$ be geometrically non-degenerate  GV-subvariety of dimension $d$ of a principally polarized complex abelian variety $(A,\theta)$. Then:

\noindent (1) (Theorem \ref{gv1}) \ $\chi_{\OO_X}(x\underline\theta)=\sum_{i=0}^d{g\choose i}(x-1)^i$.

\noindent (2) (Theorem \ref{hodgenumber}) $h^i(\OO_X)={g\choose i}$ for all $i=1,\dots, d$.

 \end{theoremalpha}

The proof of (1) is based on the study of the function $h^0_{\OO_X}(x\underline\theta)$ at the highest critical point (which turns out to be $x=1$).  (2) follows from (1) via another argument involving the Fourier-Mukai transform.

As a corollary of $(2)$, combining with the results of \cite{s} and \cite{cps}, we have

\begin{propositionalpha} \emph{(Corollary \ref{rational-singularity}).} A 2-dimensional geometrically non-degenerate GV-subscheme  is normal with rational singularities.
\end{propositionalpha}

 \noindent{\bf Application to multiplication maps of global sections of line bundles and normal generation of abelian varieties. } Finally we illustrate the interest of cohomological rank functions in another example: the ideal sheaf of one (closed) point $p\in A$. The functions
$h^i_{\I_p}(x\l)$
 seem to be highly interesting ones, especially in the perspective of basepoint-freeness criteria for primitive line bundles on abelian varieties. While we defer this to a subsequent paper, here we content ourselves to point out an elementary -- but surprising -- relation with multiplication maps of global sections of powers of line bundles.
 We consider the critical point
 \[\beta({\l})
 =\inf \{x\in\mathbb Q\>|\> h^1_{\I_p}(x\l)=0\}\]
 (as the notation suggests, such notion does not depend on $p\in A$).
 A standard argument shows that in any case $\beta(\l)\le 1$ and $\beta(\l)=1$ if and only if the polarization $\l$ has base points, i.e. a line bundle $L$ representing $\l$ (or, equivalently, all of them) has base points. Therefore, given a rational number $x=\frac{a}{b}$, it is suggestive to think that the inequality $\beta(\l)<x$ holds if and only if ``the rational polarization $x\l$ is basepoint-free". Explicitly, this means the following: let $\mu_b:A\rightarrow A$ be the multiplication-by-$b$ isogeny. Then, as it follows from the definition of cohomological rank functions, $\beta(\l)<\frac{a}{b}$ means that the finite scheme $\mu_b^{-1}(p)$ imposes independent conditions to all  translates of a given line bundle $L^{ab}$ with $c_1(L)=\l$.
  In turn $\frac{a}{b}=\beta(\l)$ means that $\mu_b^{-1}(p)$ imposes dependent conditions to a proper closed subset of translates of the line bundle $L^{ab}$ as above.\footnote{Writing $1=\frac{b}{b}$ one recovers the usual notions of basepoint-freeness and base locus.} At present we don't know how to compute, or at least bound efficiently, the invariant $\beta(\l)$ of a \emph{primitive} polarization $\l$ (except for principal polarizations of course).

Here is one of the reasons why one is lead to consider the number $\beta(\l)$. Let $\underline n$ be another polarization on $A$. We  assume that  $\underline n$ is basepoint-free. Let $N$ be a line bundle representing $\underline n$ and let $M_N$ be the kernel of the evaluation map of global sections of $N$.
We consider the critical point
\[s({\underline n})=
\inf\{x\in\mathbb Q\>|\> h^1_{M_N}(x\underline n)=0\}\]
(again this invariant does not depend on the line bundle $N$ representing $\underline n$). Well known facts about the vector bundles $M_N$ yield that,
 given $x\in\mathbb Z^+$, $s(\underline n)\le x$ if and only if the multiplication maps of global sections
\begin{equation}\label{mult} H^0(N)\otimes H^0(N^x\otimes P_\alpha)\rightarrow H^0(N^{x+1}\otimes P_\alpha)
\end{equation} are surjective for general $\alpha\in \widehat A$ and, furthermore,  $s(\n)< x$  if and only if the surjectivity holds for all $\alpha\in \widehat A$. Now the cohomological rank function leads to consider a "fractional" version of the maps (\ref{mult}). Writing $x=\frac{a}{b}$, these are the multiplication maps of global sections
\begin{equation}\label{mult-frac1}H^0(N)\otimes H^0(N^{ab}\otimes P_\alpha)\rightarrow H^0(\mu_b^*(N)\otimes N^{ab}\otimes P_\alpha)
\end{equation}
obtained by composing with the natural inclusion $H^0(N)\hookrightarrow H^0(\mu_b^*N)$. It follows
  that $s(\underline n)\le \frac{a}{b}$ if and only if the maps (\ref{mult-frac1}) are surjective for general $\alpha\in\widehat A$.
 The strict inequality holds if the surjectivity holds for all $\alpha\in \widehat A$. \footnote{Again a simple computation shows that when $x$ is an integer, writing $x=\frac{xb}{b}$ one recovers the usual notions of surjectivity of the maps (\ref{mult}) for every (resp. for general) $\alpha\in\widehat A$.}
As a simple consequence of the formulas (\ref{align}) applied to $\n=h\l$ we have
\begin{theoremalpha}\label{b-s}Let $h$ be an integer such that the polarization $h\l$ is basepoint-free \emph{(hence $h\ge 1$ if $\l$ is basepoint-free, $h\ge 2$ otherwise)}. Then
\begin{equation}\label{mult-frac}s(h {\l})= \frac{\beta({\l})}{h-\beta({\l})}\>.
\end{equation}
\end{theoremalpha}
Since $\beta(\l)\le 1$ it follows that
\[s(h {\l})\le \frac{1}{h-1}\]
and  equality holds  if and only if $\beta(\l)=1$, i.e. $\l$ has base points.

Surprisingly, this apparently unexpressive result summarizes, generalizes and improves what is known about the surjectivity of multiplication maps of global sections and projective normality  of line bundles on abelian varieties. For example, the case $h=2$ alone  tells that $s(2\l)\le 1$, with equality if and only if $\beta(\l)=1$, i.e. $\l$ has base points. In view of the above, this means that the multiplication maps (\ref{mult}) for a second power $N=L^2$ and $x=1$ are in any case surjective for general $\alpha\in\widehat A$, and in fact for all $\alpha\in\widehat A$ as soon as $\l$ is basepoint free. This is a classical result   which implies all  classical results on projective normality of abelian varieties proved via theta-groups by Mumford, Koizumi, Sekiguchi, Kempf,  Ohbuchi and others (see \cite{kempf} \S6.1-2, \cite{birke-lange} \S7.1-2 and references theiren,  see also \cite{sav} and \cite{pp1} for a theta-group-free treatment). We refer  to Section 8 below for more on this.

Finally if $\l$ is basepoint-free and $h=1$ the above Theorem tells that $\beta(\l)<\frac{1}{2}$ if and only if the multiplication maps (\ref{mult})  for $N=L$ and $x=1$ are surjective for all $\alpha\in \widehat A$. Using a well known argument, this implies

\begin{corollaryalpha} Assume that $\l$ is basepoint-free and $\beta(\l)<\frac{1}{2}$. Then $\l$ is projectively normal \emph{(this means that all line bundles $L\otimes P_\alpha$ are projective normal)}.
\end{corollaryalpha}
 This is at the same time an explanation and a generalization of Ohbuchi's theorem (\cite{oh}) asserting that, given a polarization $\n$,   $2\n$ is projectively normal
 as soon as $\n$ is basepoint-free.

 Finally, we remark that, although the applications presented in this paper concern abelian varieties and their subvarieties, the study of cohomological rank functions can be applied to the wider context of \emph{irregular varieties},
 namely varieties having non-constant morphisms to an abelian varieties, say $f:X\rightarrow A$ (as mentioned above this is indeed the point of view of the paper \cite{bps}). Given
 an element $\F\in \mathrm{D}^b(X)$, this can be done by considering the cohomological rank functions of the complex $Rf_*\F$.

\subsection*{Acknowledgements}
We thank Federico Caucci, Rob Lazarsfeld, Luigi Lombardi and  Stefan Schreieder for useful comments and suggestions. We are especially grateful to Schreieder for pointing out a gap in Section 7 of a previous version of this paper.

\section{Notation and background material}

We work  on an algebraically  closed ground field of characteristic zero.

\noindent  A polarization \emph{$\l$} on an abelian variety is the class of an ample line bundle $L$ in $\mathrm{Pic} A/\Pic0 A$. The corresponding isogeny is denoted
  \[\varphi_\l:A\rightarrow \widehat A\]
  where $\widehat A:=\Pic0 A$. For $b\in\mathbb Z$
  \[\mu_b:A\rightarrow A\qquad z\mapsto bz\]
 denotes   the multiplication-by-$b$ homomorphism.

Let $A$ be a $g$-dimensional abelian variety.
 We denote  ${\mathcal P}$ the  Poincar\'e line bundle on $A\times \widehat A$. For $\alpha\in\widehat A$ the corresponding line bundle in $A$ is denoted by $P_\alpha$,  i.e. $P_\alpha={\mathcal P}_{|A\times\{ \alpha\}}$. We always denote  $\hat e$ the origin of $\hat{A}$.

 Let $\mathrm{D}^b(A)$ be  the bounded derived category of coherent sheaves on $A$ and denote by
 \[\Phi^{A\rightarrow \widehat A}_{\mathcal P}: \mathrm{D}^b(A)\rightarrow \mathrm{D}^b(\widehat{A})\]
 the Fourier-Mukai functor associated to ${\mathcal P}$. It is an equivalence (\cite{mukai}), whose quasi-inverse is
  \begin{equation}\label{mukai-0}\Phi^{\widehat A\rightarrow A}_{{\mathcal P}^\vee[g]}: \mathrm{D}^b(\widehat A)\rightarrow \mathrm{D}^b(A)
  \end{equation}
When possible we will suppress the direction of the functor from the notation, writing simply $\Phi_{\mathcal P}$.  Since ${\mathcal P}^\vee=(-1_A,1_{\widehat A})^*{\mathcal P}=(1_A,-1_{\widehat A})^*\mathcal P$ it follows that  $\Phi_{{\mathcal P}^\vee}=(-1)^*\Phi_{\mathcal P}$.
 Finally, we will denote $R^i\Phi_{\mathcal P}$ the induced $i$-th cohomology functors.

  For the reader's convenience we list some useful facts, in use throughout the paper, concerning the above Fourier-Mukai equivalence.

 \noindent - \emph{Exchange of direct and inverse image of isogenies \emph{(\cite{mukai} (3.4))}. } Let $\varphi:A\rightarrow B$ be an isogeny of abelian varieties and let $\hat\varphi: \widehat B\rightarrow \widehat A$ be the dual isogeny. Then
 \begin{equation}\label{mukai-1}\hat\varphi^*\Phi_{{\mathcal P}_A}(\F)=\Phi_{{\mathcal P}_B}\varphi_*(\F), \qquad \hat\varphi_*\Phi_{{\mathcal P}_B}(\G)=\Phi_{{\mathcal P}_A}\varphi^*(\G)
 \end{equation}

 \noindent - \emph{Exchange of derived tensor product and derived Pontryagin product \emph{(\cite{mukai} (3.7))}. }
 \begin{equation}\label{mukai-2}
 \Phi_{\mathcal P}(\F*\G)=(\Phi_{\mathcal P}\F)\otimes (\Phi_{\mathcal P}\G)\qquad  \Phi_{\mathcal P}(\F\otimes \G)=(\Phi_{\mathcal P}\F)* (\Phi_{\mathcal P}\G)[g]
 \end{equation}

     \noindent - \emph{Serre-Grothendieck duality \emph{(\cite{mukai} (3.8). See also \cite{pp2} Lemma 2.2)}. }
  As customary, for a given  projective variety $X$ (in what follows $X$ will be either $A$ or $\widehat A$) and $\F\in \mathrm{D}^b(X)$, we denote $\F^{\vee}:=\cR Hom(\F, \OO_X)\in \mathrm{D}^b(X)$. Then
  \begin{equation}\label{mukai-3}
  (\Phi_{\mathcal P}\F)^\vee=\Phi_{{\mathcal P}^\vee}(\F^\vee)[g]
  \end{equation}

\noindent - \emph{The transform of a non-degenerate line bundle \emph{(\cite{mukai} Prop. 3.11(1))}. } Given an ample line bundle on $A$, the Fourier-Mukai transform $\Phi_\mathcal P(L)$ is a locally free sheaf (concentrated in degree $0$) on $\widehat A$,  denoted by $\widehat L$, of rank equal to $h^0(L)$. Moreover
  \begin{equation}\label{mukai-4}
  \varphi_{\l}^*\widehat L\simeq H^0(L)\otimes L^{-1}=(L^{-1})^{\oplus h^0(L)}
  \end{equation}

 \noindent - \emph{The Pontryagin product with a non-degenerate line bundle \emph{(\cite{mukai} (3.10))}. }  Given a non-degenerate line bundle $N$ on $A$, we denote $\underline n=c_1(N)$. Let $\F\in \mathrm{D}^b(A)$. Then
  \begin{equation}\label{mukai-5} \F*N=N\otimes\varphi_{\underline n}^*\bigl(\Phi_{\mathcal P}((-1)^*\F)\otimes N)\bigr)
 \end{equation}

   \noindent - \emph{(Hyper)cohomology and derived tensor product \emph{(\cite {pp2} Lemma 2.1)}. } Let $\F\in \mathrm{D}^b(A)$ and $\G\in\mathrm{D}^b(\widehat A)$.
  \begin{equation}\label{exchange} H^i(A,\F\otimes \Phi^{\widehat A\rightarrow A}_{\mathcal P}(\G))=H^i(\widehat A, \Phi^{A\rightarrow \widehat A}_{\mathcal P}(\F)\otimes \G )
\end{equation}

\section{Cohomological rank functions on abelian varieties}

In this section we define a certain non-negative rational number as the rank of the cohomology of a coherent sheaf (or, more generally, of the hypercohomology of a complex of coherent sheaves)  twisted with a rational power of a polarization. This definition is already found in \cite{barja} and, somewhat implicitly, a notion like that was already in use in \cite{kollar} (proof of Thm 17.12) and  \cite{pp3} (proof of Thm 4.1). This provides rational  cohomological rank functions satisfying certain transformation formulas under Fourier-Mukai transform (Prop. \ref{inversion} below). It follows that these functions are polynomial almost everywhere and extend to continuous functions on an open neighborhood of $\mathbb Q$ in  $\mathbb R$ (Corollaries \ref{inversion-Q} and \ref{continuity}).

\begin{definition}\label{def}
 (1) Given $\F\in \mathrm{D}^b(A)$ and $i\in \mathbb Z$, define
\[h^i_{gen}(A, \F)\]
as the dimension of hypercohomology  $H^i(A,\F\otimes P_\alpha)$, for $\alpha$ general in $\widehat A$. \footnote{It is well known that hypercohomology groups as the above satisfy the usual base-change and semicontinuity properties, see e.g. \cite{pp2} proof of Lemma 3.6 and \cite{g} 7.7.4 and Remarque 7.7.12(ii).}

\noindent (2) Given $\F\in \mathrm{D}^b(A)$,  a polarization $\l$ on $A$ and $x={\frac{a} {b}}\in\mathbb Q$, $b>0$, we define
\[h^i_{\F}(x\l)=b^{-2g}\, h^i_{gen}(A,\mu_b^*(\F)\otimes L^{ab})\]
\end{definition}

The definition is dictated from the fact that the degree of $\mu_b:A\rightarrow A$ is $b^{2g}$ (see the previous section for the notation)  and  $\mu_b^*(\l)=b^2\l$.  Therefore the pullback via $\mu_b$ of the class $\frac{a}{b}\l$ is $ab\l$. It is easy to check that the definition does not depend on the representation $x=\frac{a}{b}$. For example, if $n\in\mathbb Z$, writing $n=\frac{nb}{b}$ one gets
\[b^{-2g}h^i_{gen}((\mu_b^*\F)\otimes L^{b^2n})=b^{-2g}h^i_{gen}(\mu_b^*(\F\otimes L^n))=b^{-2g}\sum_{\alpha\in \hat\mu_b^{-1}(\hat e)}h^i_{gen}(\F\otimes L^n\otimes P_\alpha)=h^i_{gen}(\F\otimes L^n)\]
where $\hat e$ is the identity point of $\widehat A$ and $\hat\mu_b:\widehat A\rightarrow \widehat A$ is the dual isogeny.

\begin{remark}\label{Q-twisted}[Coherent sheaves $\mathbb Q$-twisted by a polarization] Let $\l$ be a polarization on our abelian variety $A$. Following Lazarsfeld (\cite{laz2}), but somewhat more restrictively,  we will define \emph{ coherent sheaves $\mathbb Q$-twisted by $\l$} as equivalence classes of pairs $(\F,x\l)$ where $\F$ is a coherent sheaf on $A$ and $x\in\mathbb Q$, with respect to the equivalence relation generated by $(\F\otimes L^h,x\l)\sim (\F,(h+x)\l)$, for $L$ a line bundle representing $\l$ and $h\in\mathbb Z$. Such thing is denoted $\F\langle x\l\rangle$ (note  that\  $\F\otimes P_\alpha\langle x\l\rangle=\F\langle x\l\rangle$\  for $\alpha\in \widehat A$). Similarly, one can define complexes of coherent sheaves  $\mathbb Q$-twisted by the polarization $\l$.
 Now the quantity $h^i_\F(x\l)$ depends only on the $\mathbb Q$-twisted complex $\F\langle x\l\rangle$ and one may think of it as the (generic) cohomology rank
$h^i(A,\F\langle x\l\rangle)$.
\end{remark}

Some immediate basic properties of   generic cohomology ranks defined above are:

\noindent (a)   \  $\chi_\F(x\l)=\sum_i (-1)^ih^i_{\F}(x\l)$, \
where $\chi_\F(x\l)$ is the Hilbert polynomial, i.e. the Euler characteristic.

\noindent (b) Serre duality: $h^i_{\F}(x\l)=h^{g-i}_{\F^{\vee}}(-x\l)$.

\noindent (c) Serre vanishing: \emph{given a coherent sheaf $\F$ there is a $x_0\in \mathbb Q$ such that $h^i_\F(x\l)=0$ for all $i>0$ and for all rational $x\ge x_0$. }

\noindent \emph{Proof of} (c). It is well known that there is  $n_0\in \mathbb Z$ such that $h^i(A,\F\otimes L^{n_0}\otimes P_\alpha)=0$ for all $i>0$ and for all $\alpha\in\Pic0 X$. Following the terminology of Mukai, this condition is referred to as follows: \emph{$\F\otimes L^{n}$ satisfies IT(0)} (the Index Theorem with index 0, see also \S5 below). Therefore, for all $b\in \mathbb Z^+$, $ \mu_b^*(\F)\otimes L^{b^2n_0}$ satisfies IT(0). The tensor product of a coherent IT(0) sheaf with a locally free IT(0) sheaf is IT(0) (see e.g. \cite[Prop. 3.1]{pp3} for a stronger result). Therefore $\mu_b^*(\F)\otimes L^m$ satisfies IT(0) for all $b\in \mathbb Z^+$ and $m\ge b^2n_0$. This is more than enough to ensure that $h^i_{\F}(x\l)=0$ for all rational numbers $x\ge n_0$.\footnote{More precisely this proves, in the terminology of Section 5 below, that the $\mathbb Q$-twisted coherent sheaves $\F\langle x\l\rangle$ satisfy IT(0) for all $x\in\mathbb Q^{\ge n_0}$.} \endproof

The following Proposition describes the behavior of the generic cohomology ranks with respect to the Fourier-Mukai transform.

\begin{proposition}\label{inversion-a}  Let $\F\in \mathrm{D}^b(A)$ and let $\l$ be a polarization on $A$. Then, for $x\in\mathbb Q^+$
\[h^i_{\F}(x\l) = \frac{x^g}{\chi(\l)}h^{g-i}_{\varphi_\l^*\Phi_{\mathcal P^\vee}(\F^\vee)}(\frac{1}{x}\l)\]
and, for $x\in\mathbb Q^-$,
 \[h^i_{\F}(x\l) =
 \frac{ (-x)^g}{\chi(\l)}h^i_{\varphi_\l^*\Phi_{\mathcal P}(\F)}(-\frac{1}{x}\l)\]
\end{proposition}
\begin{proof}
Let us start with the case $x=\frac{a}{b}\in\mathbb Q^+$. Then,
\begin{eqnarray*}
h^i_{\F}(x\l)=\frac{1}{b^{2g}}h^i_{gen}(A, \mu_b^*\F\otimes L^{ab})&=&\frac{1}{b^{2g}}\dim \mathrm{Ext}^i_A(\mu_b^*\F^{\vee}, L^{ab}_{\alpha})\\&=&\frac{1}{b^{2g}}\dim \mathrm{Ext}^i_{\hat{A}}(\Phi_{\mathcal P}(\mu_b^*\F^{\vee}),\Phi_{\mathcal P}( L^{ab}_{\alpha})),
\end{eqnarray*}
where $\alpha\in\hat{A}$ is  general, $L^{ab}_{\alpha}:=L^{ab}\otimes P_{\alpha}$, and the last equality holds by Mukai's equivalence \cite{mukai}.

Note that, by  (\ref{mukai-1}), $\Phi_{\mathcal P}(\mu_b^*\F^{\vee})\simeq \hat{\mu}_{b*}\Phi_{\mathcal P}(\F^{\vee})$
 where $\hat{\mu}_{b}: \widehat{A}\rightarrow \widehat{A}$ is the multiplication by $b$ on $\widehat{A}$. By (\ref{mukai-4}) $R\Phi_{\mathcal P}( L^{ab}_{\alpha}):=\widehat{L^{ab}_{\alpha}}$ is a vector bundle on $\widehat{A}$ and
 \begin{equation}\label{to-be-inserted}\mu_{ab}^*\varphi_\l^*\widehat{L^{ab}_{\alpha}}=\varphi_{ab\l}^*\widehat{L^{ab}_{\alpha}}
 \simeq ((L^{ab}_{\alpha})^{-1})^{\oplus h^0(L^{ab})}.
 \end{equation}
 Hence, for general $\alpha\in\widehat A$,
\begin{eqnarray*}
h^i_{\F}(x\l)&=&\frac{1}{b^{2g}}\dim \mathrm{Ext}^i_{\widehat{A}}(\hat{\mu}_{b*}\Phi_{\mathcal P}(\F^{\vee}), \widehat{L^{ab}_{\alpha}})=\frac{1}{b^{2g}}\dim \mathrm{Ext}^i_{\widehat{A}}(\Phi_{\mathcal P}(\F^{\vee}), \hat{\mu}_{b}^*\widehat{L^{ab}_{\alpha}})\\
&=&\frac{1}{b^{2g}}\dim \mathrm{Ext}^{g-i}_{\widehat{A}}(\hat{\mu}_{b}^*\widehat{L^{ab}_{\alpha}},\Phi_{\mathcal P}(\F^{\vee}) )\\&=&\frac{1}{\deg \hat{\mu}_a\deg \varphi_\l}\frac{1}{b^{2g}}\dim \mathrm{Ext}^{g-i}_{A}(\varphi_
\l^*\hat{\mu}_a^*\hat{\mu}_{b}^*\widehat{L^{ab}_{\alpha}},\varphi_\l^*\hat{\mu}_a^*\Phi_{\mathcal P}(\F^{\vee}) )\\&=&\frac{1}{\chi(\l)^2}\frac{1}{a^{2g}b^{2g}}\dim \mathrm{Ext}^{g-i}_{A}(\varphi_{ab\l}^*\widehat{L^{ab}_{\alpha}},\mu_a^*\varphi_\l^*\Phi_{\mathcal P}(\F^{\vee}) )\\&\buildrel{(\ref{to-be-inserted})}\over =&\frac{1}{\chi(\l)}\frac{1}{a^{g}b^{g}}h^{g-i}(A,\mu_a^*\varphi_\l^*\Phi_{\mathcal P}(\F^{\vee}) \otimes L^{ab}_{\alpha}).
\end{eqnarray*}
We also note that $(-1)_{\widehat{A}}^*\Phi_{\mathcal P}(\F^{\vee})=\Phi_{\mathcal{P}^{\vee}}(\F^{\vee})$ and $(-1)_{{A}}^*\l=\l$. Hence, applying $(-1)_A^*$, we get
\begin{eqnarray*}h^i_{\F}(x\l)\> =\> \frac{1}{\chi(\l)}\frac{1}{a^{g}b^{g}}h^{g-i}_{gen}(\mu_a^*\varphi_\l^*\Phi_{\mathcal{P}^{\vee}}(\F^{\vee}) \otimes L^{ab})\> =\>
\frac{1}{\chi(\l)}\frac{a^g}{b^g}h^{g-i}_{\varphi_\l^*\Phi_{\mathcal{P}^{\vee}}(\F^{\vee})}(\frac{b}{a}\l)).\end{eqnarray*}
By similar argument (or by Serre duality) we get the equalities when $x\in\mathbb Q^-$.
\end{proof}

\begin{corollary}\label{inversion} Under the same hypothesis and notation of the previous Proposition,  for each $i\in\mathbb Z$ there are  $\epsilon^-,\epsilon^+ >0$ and two polynomials $P^-_{i,\F},P^+_{i,\F}\in\mathbb Q [x]$ of degree $\le \dim A$ such that, for $x\in (-\epsilon^-,0)\cap \mathbb Q$,
\[h^i_\F(x\l)\> =\> P^-_{i,\F}(x) \]
and, for $x\in (0,\epsilon^+)\cap \mathbb Q$
\[h^i_\F(x\l)\> =\> P^+_{i,\F}(x)\]
More precisely
\[h^i_{\F}(x\l) =
 \frac{(-x)^g}{\chi(\l)}\chi_{\varphi_\l^*R^i\Phi_{\mathcal P}(\F)}(-\frac{1}{x}\l) \quad\hbox{\emph{for} $x\in (-\epsilon^-,0)\cap \mathbb Q$}\]
\[h^i_{\F}(x\l)=
 \frac{x^g}{\chi(\l)}\chi_{\varphi_\l^*R^{g-i}\Phi_{\mathcal P^\vee}(\F^\vee)}(\frac{1}{x}\l)\quad\hbox{\emph{for} $x\in(0,\epsilon^+)\cap \mathbb Q$}\]
\end{corollary}

\proof The statement follows from Proposition \ref{inversion} via Serre vanishing (see (c) above in this section). Indeed for a sufficiently small $x\in\mathbb Q^+$ we have that $h^k_{\varphi_\l^*R^{j}\Phi_{\mathcal P^\vee}(\F^\vee)}(\frac 1 {x} \l))=0$ for all $k\ne 0$ and all $j$. Therefore the hypercohomology spectral sequence computing $ h^{g-i}_{\varphi_\l^*\Phi_{\mathcal P^\vee}(\F^\vee)}(\frac 1 {x} \l))$ \footnote{this means the spectral sequence computing the hypercohomogy groups $H^{g-i}(A,\mu_a^*(\varphi_\l^*\Phi_{\mathcal P^\vee}(\F^\vee))\otimes L^{ab}\otimes P_\alpha)$ for  $\frac{ b}{a}=\frac{1}{x}$ and $\alpha\in\Pic0 X$ general} collapses so that
\[ h^{g-i}_{\varphi_\l^*\Phi_{\mathcal P^\vee}(\F^\vee)}(\frac 1 {x}\l))=h^0_{\varphi_\l^*R^{g-i}\Phi_{\mathcal P^\vee}(\F^\vee)}(\frac 1 {x} \l))=\chi_{\varphi_\l^*R^{g-i}\Phi_{\mathcal P^\vee}(\F^\vee)}(\frac 1 {x}\l))\>.\]
 This proves the statement for $x>0$. The proof for the case $x<0$ is the same.
\endproof

\begin{remark}\label{eff-serre} It follows from the proof that one can take as $\epsilon^-$ the minimum, for all $i$, of $\frac{1}{x_i}$, where $x_i$ is a bound ensuring Serre vanishing for twists with powers of $L$  of the sheaf $\varphi_\l^*R^i\Phi_{\mathcal P}(\F)$. Similarly for $\epsilon^+$.
\end{remark}

The next Corollary shows that the statement of the previous Corollary holds more generally in $\mathbb Q$-twisted setting.

\begin{corollary}\label{inversion-Q} Same hypothesis and notation of the previous Proposition. Let $x_0\in\mathbb Q$. For each $i\in\mathbb Z$ there are  $\epsilon^-,\epsilon^+ >0$ and two polynomials $P^-_{i,\F,x_0},P^+_{i,\F,x_0}\in\mathbb Q [x]$ of degree $\le  \dim A$ such that, for $x\in (x_0-\epsilon^-,x_0)\cap \mathbb Q$,
\[h^i_\F(x\l)\> =\> P^-_{i,\F,x_0}(x) \]
and, for $x\in (x_0,x_0+\epsilon^+)\cap \mathbb Q$
\[h^i_\F(x\l)\> =\> P^+_{i,\F,x_0}(x)\]
\end{corollary}
\proof This follows by reducing to the previous Corollary via the formula
\[h^i_\F((x_0+y)\l)=b^{-2g}h^i_{\mu_b^*(\F)\otimes L^{ab}}(b^2y\l)\]
for $x_0=\frac{a}{b}$, $b>0$.
\endproof

As a consequence we have
\begin{corollary}\label{continuity}
 The functions $h^i_{\F,\l}:\mathbb Q\rightarrow \mathbb  Q^{\ge 0}$  extend to a continuous functions $h^i_{\F,L}: U\rightarrow \mathbb R$, where $U$ is an open subset of $\mathbb R$ containing $\mathbb Q$, satisfying the condition of Corollary \ref{inversion-Q} above, namely for each $x_0\in U$ there exist $\epsilon^-,\epsilon^+>0$ and two polynomials $P_{i,\F,x_0}^-,P_{i,\F,x_0}^+\in\mathbb Q[x]$ of degree $\le \dim A$, having the same value at $x_0$,  such that
 \[h^i_{\F}(x_\l)=\begin{cases}P^-_{i,\F,x_0}(x)&\hbox{for $x\in (x_0-\epsilon^-,x_0]$}\\
 P^+_{i,\F,x_0}(x)&\hbox{for $x\in [x_0,x_0+\epsilon^+)\>.$}
 \end{cases}\]
\end{corollary}
\proof Let $x_0\in \mathbb Q$.
The assertion follows from  Corollary \ref{inversion-Q} because, for $x_0\in\mathbb Q$, $P^-_{i,\F,x_0}(x_0)=P^+_{i,\F,x_0}(x_0)=h^i_{\F}(x_0\l)$.
 To prove this we can assume, using Corollary \ref{inversion-Q}, that $x_0=0$. By Corollary \ref{inversion} we have that $P^-_{i,\F,x_0}(0)$ (resp. $P^+_{i,\F,x_0}(0)$) coincide, up the the same multiplicative constant, with the coefficients of degree $g$ of the Hilbert polynomial of the sheaves $\varphi_\l^*R^i\Phi_{\mathcal P}(\F)$, resp. $\varphi_\l^*R^{g-i}\Phi_{\mathcal P^\vee}(\F^\vee)$. Hence they coincide, up to the same multiplicative constant, with the generic ranks of the above sheaves. By cohomology and base change, and Serre duality,  such generic ranks coincide with  $h^i_{\F,\l}(0)$.
 \endproof

\begin{remark}\label{acc} It seems likely that $U=\mathbb R$, hence the cohomological rank functions would be piecewise-polynomial (compare \cite[Question 8.11]{bps}). This would follow from the absence of accumulation points in $\mathbb R\smallsetminus U$, but at present we don't know how to prove that. In any case, in the next section we prove that the cohomological rank functions extend to continuous function on the whole $\mathbb R$.
\end{remark}

 Given two objects $\F$ and $\G$ in $\mathrm{D}^b(A)$ and $f\in\text{Hom}_{\mathrm{D}^b(A)}(\F,\G)$ one can   define similarly the $i$-th cohomological rank, nullity and corank of the maps
 twisted with a rational multiple of a polarization $\l$ as the generic rank, nullity and corank of the maps
 \[H^i(A, \mu_b^*(\F)\otimes L^{ab}\otimes P_\alpha)\rightarrow H^i(A,\mu_b^*(\G)\otimes L^{ab}\otimes P_\alpha).\]
This gives rise to functions \  $\mathbb Q\rightarrow \mathbb Q^\ge 0$ \  satisfying the same properties.
Let us consider, for example, the rank,  (the kernel and the corank have completely similar description) and let us denote it $rk\bigl(h^i_f(x\l)\bigr)$.
 \begin{proposition}\label{functions}
 Let $x_0\in\mathbb Q$. For each $i\in\mathbb Z$ there are  $\epsilon^-,\epsilon^+ >0$ and two polynomials $P^-_{i,f,x_0},P^+_{i,f,x_0}\in\mathbb Q [x]$ of degree $\le \dim A$ such that, for $x\in (x_0-\epsilon^-,x_0)\cap \mathbb Q$,
\[rk\bigl(h^i_f(x\l)\bigr)\> =\> P^-_{i,f, x_0}(x) \]
and, for $x\in (x_0,x_0+\epsilon^+)\cap \mathbb Q$
\[rk\bigl(h^i_f(x\l)\bigr)\> =\> P^+_{i,f,x_0}(x).\]
 \end{proposition}
\begin{proof} As above, we can assume that $x_0=0$. By Corollary \ref{inversion} and its proof there is a $\epsilon^->0$ such that for $x=\frac{a}{b}\in (-\epsilon^-,0)$  (with $a<0$ and $b>0$), $rk\bigl(h^i_f(x\l)\bigr)$ coincides with
$\frac{(-x)^g}{\chi(\l )} rk (F_{-\frac{1}{x}})$
 where $F_{-\frac{1}{x}}$ is the natural map
\[ F_{-\frac{1}{x}}: H^0\bigl(\mu_{-a}^*(\varphi_\l^*R^i\Phi_{\mathcal P}\F)\otimes L^{-ab}\bigr)\rightarrow H^0\bigl(\mu_{-a}^*(\varphi_\l^*R^i\Phi_{\mathcal P}\G)\otimes L^{-ab}\bigr)  \]
By an easy calculation with Serre vanishing (see (c) in this Section), up to taking a smaller $\epsilon^-$ the image of the map $F_{-\frac{1}{x}}$ is $H^0$ of the image of the map coherent sheaves
\[\mu_{-a}^*(\varphi_\l^*R^i\Phi_{\mathcal P}\F)\otimes L^{-ab} \rightarrow \mu_{-a}^*(\varphi_\l^*R^i\Phi_{\mathcal P}\G)\otimes L^{-ab}\]
and its dimension is
\[\chi_{Im(\varphi_\l^*R^i\Phi_{\mathcal P}(f))}(-\frac{1}{x}\l)\]
In conclusion, for $x\in (-\epsilon^-,0)$
\[rk\bigl(h^i_f(x\l)\bigr)=\frac{(-x)^g}{\chi(\l)}\chi_{Im(\varphi_\l^*R^i\Phi_{\mathcal P}(f))}(-\frac{1}{x}\l):=P^-_{i,f,0}(x)\]
Similarly, for $x\in (0,\epsilon^+)$
\[rk\bigl(h^i_f(x\l)\bigr)=\frac{(x)^g}{\chi(\l)}\chi_{Im(\varphi_\l^*R^{g-i}\Phi_{\mathcal P^\vee}(f))}(\frac{1}{x}\l):=P^+_{i,f,0}(x).\]
\end{proof}

\section{Continuity as real functions}

The aim of this section is to prove Theorem \ref{c0} below, asserting that the cohomological rank functions extend to continuous functions on the whole $\mathbb R$  (see Remark \ref{acc}).  We
 start with a version of Serre's vanishing needed in the proof.

 \begin{lemma}\label{effectivebound} Let $A$ be an abelian variety and let $L$ be a very ample line bundle on $A$. Let $\F$ be a  coherent sheaf of dimension $n$ on $A$. There exist two integers $M^-$ and $M^+$ such that for all integers $m\in \mathbb Z^+$, for all $k=0,\dots n$ and  all sufficiently general complete intersections $Z_k=D_1\cap D_2\cap\cdots \cap D_k$ of $k$ divisors  $D_i\in |m^2\rho_iL|$ with $0<\rho_i<1$ rational with $m^2\rho_i\in\mathbb Z$ \emph{(here we understand $Z_0=A$)}, and for all $\alpha\in\widehat A$, the following conditions hold:
\[\begin{cases}h^i(\mu_m^*\F\mid_{Z_k}\otimes L^{m^2t}\otimes P_\alpha)=0&\hbox{for all \ \ $i\geq 1$ and $t\in \mathbb Z^{\ge M^+}$}\\
h^i(\mu_m^*\F\mid_{Z_k}\otimes L^{m^2t}\otimes P_\alpha)=\chi(\mathcal{E}xt^{g-i}(\mu_m^*\F\mid_{Z_k}, \mathcal O_A)\otimes L^{-m^2t})&\hbox{ for all  $i\leq n-1-s$ and $t\in \mathbb Z^{\leq M^-}$}
\end{cases}\]
 The pair $(M^-,M^+)$ will be referred to as \emph{an effective cohomological bound for $\F$}.
 \end{lemma}
 \begin{proof} Note that the statement makes sense since $L$ is assumed to be very ample and $\tau_i:=m^2\rho_i\in \mathbb Z^+$.
 Since $Z_k$ is a general complete intersection, the Koszul resolution of $\OO_{Z_k}$, tensored with $\mu_m^*\F$
 \begin{equation}\label{koszul}0\rightarrow \mu_m^*F\otimes L^{-\sum_i\tau_i} \rightarrow \cdots\rightarrow \mu_m^*\F\otimes(\oplus_i L^{-\tau_i})\rightarrow \mu_m^*\F\rightarrow \mu_m^*F\mid_{Z_k}\rightarrow 0
 \end{equation}
 is exact. Therefore the bound of the upper line is a variant of Serre vanishing, in the version of the previous section, via a standard diagram-chase.

 Concerning the lower line, we first prove it for $k=0$. By Serre duality
\begin{eqnarray*}H^i(\mu_m^*\F\otimes L^{m^2t}\otimes P_\alpha)=H^{g-i}((\mu_m^*\F)^\vee\otimes L^{-m^2t}\otimes P_\alpha^{\vee} )
= H^{g-i}(\mu_m^*(\F^\vee)\otimes L^{-m^2t}\otimes P_\alpha^{\vee} ).\end{eqnarray*}

Since $\F$ is a  coherent sheaf of dimension $n$, $\mathcal{E}xt^j(F, \mathcal{O}_A)$ vanishes for $j<g-n$ while it has
 codimension $\geq j$ with support  contained in the support of $\F$ for $j\ge g-n$ (see for instance \cite[Proposition 1.1.6]{hl}).  We apply Serre vanishing to find an integer $N$ such that
  such that
  \[H^j(\mu_m^*\mathcal{E}xt^{i}(\F, \mathcal{O}_A)\otimes L^{-m^2t}\otimes P_\alpha^{\vee})=0\quad\hbox{  for all $t\in \mathbb Z$ such that $-t\ge N$ and $j\geq 1$.}\]
   The statement of the bottom line for $k=0$ follows via the  spectral sequence
   \[H^{h}(\mu_m^*{\mathcal Ext}^{g-i-h}(\F,\OO_A)\otimes L^{-m^2t}\otimes P_\alpha^\vee)\Rightarrow H^{g-i}((\mu_m^*\F)^\vee\otimes L^{-m^2t}\otimes P_\alpha^{\vee} ).\]
   At this point the statement of the lower line for all $k\le g-n$ follows as above from the case $k=0$  and the fact  that for a general choice of a very ample divisor $D$  we have a short exact sequence for all $j\geq 0$ $$0\rightarrow\mu_m^* \mathcal Ext^j(\F, \mathcal O_A)\rightarrow\mu_m^* \mathcal Ext^j(\F, \mathcal O_A)\otimes \mathcal{O}_A(D)\rightarrow \mathcal Ext^{j+1}(\mu_m^*\F\mid_D, \mathcal{O}_A)\rightarrow 0.$$
   (see \cite[Lemma 1.1.13]{hl}).
 \end{proof}

\begin{theorem}\label{c0} Let $A$ be an abelian variety, let $\l$ be a polarization on $A$ and  $\F\in \mathrm{D}^b(A)$. The functions $x\mapsto h_{\mathcal{F}}^i(x\l)$ extend to  continuous functions on $\mathbb R$.
Such functions are bounded above by a polynomial function of degree at most $n=\dim \F$, whose coefficients involve only the intersection numbers of the support of $\F$ with powers of $L$, the ranks of the cohomology sheaves of $\F$ on the generic points of their support, and an effective bound $(N^-, N^+)$ of generic cohomology of $\F$.
\end{theorem}
\begin{proof} We can assume that $\l$ is very ample.  The proof will be in some steps. To begin with, we prove the statement under the assumption that $\F$ is a pure sheaf.
  Let $V_1,\ldots, V_s$ be the irreducible components of the support of $\F$ with reduced scheme structures. Hence each $V_j$ is an integral variety. Let $t_j$ be the length of $\F$ at the generic point of $V_j$ and define
 $$u(\F):=\sum_jt_j(V_j\cdot L^n)_A.$$  We have seen in the previous lemma that $h_{\F}^i(x\l)$ are natural polynomial functions for $x\leq M^-$ and $x\geq M^+$. We now deal with the case when $M^-\leq x\leq M^+$. More precisely we will prove, by induction on $n=\dim \F$, the following statements
 \begin{itemize}
 \item[(a)] $h^i_{\F, \l}$ extends to  a continuous function on $\mathbb R$;
 \item[(b)] $h^0_{\F, \l}(x)\leq \frac{u(\F)}{n!}(x-M^-)^n$, for $x\geq M^-$; \\$h^i_{\F, \l}(x)\leq 2^{n-1}u(\F)(M^+-M^-)^{n-1}$, for $M^-\leq x\leq M^+$,\\ and $h^n_{\F, \l}(x)\leq \frac{u(\F)}{n!}(M^+-x)^n$, for $x\leq M^+$.
 \end{itemize}
 These assertions are clear if $\dim\F=0$.
  Assume that they hold for all pure sheaves of dimension $\leq n-1$. We will prove that they imply the following assertions:

\noindent\emph{ For all pure sheaves $\F$ with $\dim\F=n$ and for all rational numbers $x$ and $0<\epsilon<1$}
 \begin{eqnarray}\label{derivativebound}
\label{1}&& h_{\F}^0((x+\epsilon)\l)-h_{\F}^0(x\l)\leq \epsilon\frac{u(\F)}{(n-1)!}(x+\epsilon-M^-)^{n-1}\; \mathrm{for}\; x\geq M^-;\\
\label{2}&&| h_{\F}^i((x+\epsilon)\l)-h_{\F}^i(x\l)|\leq \epsilon 2^{n-1}u(\F)(M^+-M^-)^{n-1}\; \mathrm{for}\; M^-\leq x<x+\epsilon\leq M^+;\\
\label{3}&& h_{\F}^{n}(x\l )-h_{\F}^{n}((x+\epsilon)\l)\leq \epsilon\frac{u(\F)}{(n-1)!}(M^+-x-\epsilon)^{n-1}\; \mathrm{for}\;  x+\epsilon\leq M^+.
 \end{eqnarray}

   Take $M$ sufficiently large and divisible such that $Mx$ and $M\epsilon$ are integers, take a general divisor $D\in |M^2\epsilon L|$ and consider the short exact sequence:
$$0\rightarrow \mu_M^*\mathcal F\otimes L^{M^2x} \xrightarrow{\cdot D} \mu_M^*\mathcal F\otimes L^{M^2(x+\epsilon)}\rightarrow  \mu_M^*\mathcal F\otimes L^{M^2(x+\epsilon)}\mid_D\rightarrow 0.$$
Taking the long exact sequence of cohomology of the above sequence tensored with a general $P_\alpha\in \widehat A$, we see that
\begin{equation}
h_{\F}^0((x+\epsilon)\l)-h_{\F}^0(x\l)\leq \frac{1}{M^{2g}}h^{0}_{gen}( \mu_M^*\mathcal F\otimes L^{M^2(x+\epsilon)}\mid_D)
=\frac{1}{M^{2g}}h^0_{\mu_M^*\mathcal{F}\mid_D}(M^2(x+\epsilon)\l) .
 \end{equation}
Note that $\mu_M^*\F\mid_D$ is a pure sheaf on $A$ of dimension $n-1$. It is also easy to see that an effective  cohomological bound of $\mu_M^*\F\mid_D$ is $(M^2M^-, M^2M^+)$. Hence condition $\text{(b)}_{n-1}$ above yields that
\[h_{\mu_M^*\F\mid_D}^0(M^2(x+\epsilon)\l)\leq \frac{u(\mu_M^*\mathcal{F}\mid_D)}{(n-1)!}M^{2n-2}(x+\epsilon-M^-)^{n-1}.\]
The components of the support of $\mu_M^*\F\mid_D$ are $\mu^{-1}V_1\cap D,\ldots, \mu^{-1}V_s\cap D$. Hence
\[u(\mu_M^*\mathcal{F}\mid_D)=\sum_jt_j(\mu_M^{-1}V_j\cdot D\cdot L^{n-1})_A=\frac{\epsilon}{M^{2n-2}}\sum_jt_j(\mu_M^{*}V_j\cdot \mu_M^*L^{n})_A\\= \epsilon M^{2g-2n+2}r(\F).\]
It follows that
\[h_{\F}^0((x+\epsilon)\l)-h_{\F}^0(x\l)\leq \epsilon\frac{u(\F)}{(n-1)!}(x+\epsilon-M_0)^{n-1}\]
i.e. $(\ref{1})_n$.
The estimate $(\ref{3})_n$ is proved exactly in the same way as the $h_{\F, L}^0$ case. Concerning $(\ref{2})_n$  note that for $M^-\leq x<x+\epsilon\leq M^+$,
\begin{eqnarray*} &&|h_\F^i((x+\epsilon)\l)-h_{\F}^i(x\l)|\\ &\leq &\frac{1}{M^{2g}} \big(h_{gen}^{i-1}(\mu_M^*\F\mid_{D}\otimes L^{M^2(x+\epsilon)})+h_{gen}^{i}(\mu_M^*\mathcal F\mid_{D}\otimes L^{M^2(x+\epsilon)})\big)\\
&\leq &  \epsilon 2^{n-1}u(\F)(M^+-M^-)^{n-1}.
\end{eqnarray*} This concludes the proof of the estimates (\ref{1}) (\ref{2}) (\ref{3}) under the assumption that $(b)_{n-1}$ holds.

Turning to $(a)_n$ and $(b)_n$,  note that the functions $h_{\F}^i(x\l)$ satisfy the statement of Corollaries \ref{inversion-Q} and \ref{continuity}. Therefore
the left derivative $D^-h_{\F}^i(x\l)$ and the right derivative $D^+h_{\F}^i(x\l)$ exist on all $x\in\mathbb Q$ (in fact on all $x\in U$ of Cor. \ref{continuity}), and they coincide away of a discrete subset. The inequalities $(\ref{1})_n$, $(\ref{2})_n$ and $(\ref{3})_n$ show that both derivatives are bounded above by the corresponding polynomials of degree $n-1$. Note that by Lemma \ref{effectivebound} and the assumption that $\F$ is pure, we have $h^0_{\F}(M^-\l)=0$  and $h^i_{\F}(M^+\l)=0$ for $i\geq 1$, and hence by integration, the above bounds for derivatives imply $(a)_n$ and $(b)_n$ for all $x\in U$ as above and hence, by continuity, for all $x\in\mathbb R$.
This concludes the proof of the Theorem for pure sheaves.

 Next, we prove the Theorem for all coherent sheaves $\F$ on $A$. Assume that  $\dim \F=n$.
  We consider the torsion filtration of $\mathcal F$:
\[T_0(\mathcal F)\subset T_1(\mathcal F)\subset\cdots\subset T_{n-1}(\mathcal F)\subset T_n(\mathcal F)= \mathcal F,\]
 where $ T_i(\mathcal F)$ is the maximal subsheaf of $\mathcal F$ of dimension $i$ and hence $\mathcal Q_i:=T_i(\mathcal F)/T_{i-1}(\mathcal F)$ is a pure sheaf of dimension $i$.
 We see that $h^0_{\F}(x\l)\leq h^0_{T_{n-1}(\F)}(x\l)+h^0_{\mathcal Q_n}(x\l)$ and we also have, adopting the previous notation,
 \begin{eqnarray*}h^0_{\F}((x+\epsilon)\l)-h^0_{\F}(x\l)&\leq &\frac{1}{M^{2g}}h^0_{\mu_M^*\F\mid_D}(M^2(x+\epsilon)\l)
 \\ &\leq &\frac{1}{M^{2g}}\big(h^0_{\mu_M^*T_{n-1}(\F)\mid_D}(M^2(x+\epsilon)\l)+h^0_{\mu_M^*\mathcal{Q}_n\mid_D}(M^2(x+\epsilon)\l)\big).
 \end{eqnarray*}
We then proceed by induction on $\dim \F$  and the results on the pure sheaf case to prove the continuity of the function $h^0_{\F}(x\l)$ and its boundedness. The proof of continuity for other cohomology rank functions of $\F$ is similar.

Similarly the proof of the statement of the Theorem for objects of the bounded derived category follows the same lines, using the functorial hypercohomology spectral sequences
\[E_2^{h,k}(\mu^*_m(\F)\otimes L^r\otimes P_\alpha):=H^h(A, \mathcal H^k(\mu_m^*\F\otimes L^r\otimes P_\alpha))\Rightarrow H^{h+k}(A, \mu_m^*\F\otimes L^r\otimes P_\alpha)\]
This time, for $x\in\mathbb Q$, $x=\frac{a}{b}$ with $b>0$ one defines the cohomological rank functions for the groups appearing at each page: $e_{r,\F}^{h,k}(x\l):=b^{-2g}
\dim E_r^{h,k}(\mu^*_b(\F)\otimes L^a\otimes P_\alpha)$ for general $\alpha\in \widehat A$. Using Proposition \ref{functions} these  functions are already defined in $\mathbb R$ minus a discrete set satisfying the property stated in  Corollary \ref{inversion-Q}. By induction on $r$ and on the dimension of the cohomology sheaves one proves that these functions can be extended to continuous functions  satisfying the same property. They are bounded as above. From this and the convergence one gets the same statements for the functions $h^i_\F(x\l)$. We leave the details to the reader.
 \end{proof}

\section{Critical points and jump loci}

A \emph{critical point for the function $x\mapsto h^i_\F(x\l)$}  is a $x_0\in\mathbb R$ where the function is not smooth.  We denote $S_{\F,\l}^i$ the set of critical points of $h_{\F}^i(x\l)$ and let $S_{\F, \l}=\cup_i S_{\F,\l}^i$.   This is the subject of this section.   In all examples we know, the critical points of a cohomological rank function are finitely many, and satisfy the conclusion of Corollary \ref{continuity}. We expect this to be true in general.

 It follows  from the results of Section 2 that for $x_0\in\mathbb Q$, or more generally for $x_0$ in the open set $U$ of Corollary \ref{continuity},   $x_0$ is a critical point if and only if the polynomials $P^-_{i,\F,x_0}$ and $P^+_{i,\F,x_0}$ do not coincide.  As we will see below it is easy to produce examples of rational critical points. However they can be irrational -- even for line bundles  on abelian varieties -- as shown by the following example.
\begin{example}\label{irrational} Let $(A,\l)$ be a polarized abelian variety and let  $M$ be a non-degenerate line bundle on $A$.
Consider the polynomial $P(x)=\chi_M(x\l)$. By Mumford, all roots of $P(x)$ are real numbers. Denote them by $\lambda_1>\lambda_2>\cdots>\lambda_k$, let $m_i$ be the multiplicity of $\lambda_i$ and finally denote by $\lambda_0=\infty$, $m_0=0$, and $\lambda_{k+1}=-\infty$. We know by \cite[Page 155]{mumford} that for any rational number $x=\frac{a}{b}\in (\lambda_{i+1}, \lambda_i)$ with $b>0$, the line bundle $M^{b}\otimes L^a$  is non-degenerate and has index $a_i:=\sum_{0\leq k\leq i}m_k$.  Hence
 \[h_{M}^{k}(x\l)=\begin{cases}(-1)^{k}P(x)&\hbox{if $k=a_i$ for some $0\leq i\leq k$, and $x\in (\lambda_{i+1}, \lambda_i)  $}\\
0&\hbox{ otherwise}\end{cases}\]
 Hence the set of critical points of $M$ is $\{\lambda_1,\ldots,\lambda_k\}$.
If $A$ is a simple abelian variety and $M$ and $L$ are linearly independent in $\mathrm{NS}(A)$ (such abelian varieties exist and they are called Shimura-Hilbert-Blumenthal varieties, see for instance \cite{dl}), then all roots of $P(x)$ are irrational. Actually, if some root $\lambda_i=\frac{a}{b}$ is rational, then $M^b\otimes L^a$ is a non-trivial degenerate line bundle and hence its kernel is a non-trivial abelian subvariety of $A$, which is a contradiction.
 \end{example}

A critical point $x_0\in\mathbb R$ is  said to be  \emph{of index $k$}  if the function $h^i_\F(x\l)$ is of class $\mathcal C^{k}$ but not $\mathcal C^{k+1}$ at $x_0$. By Corollaries \ref{inversion-Q} and \ref{continuity} if $x_0\in\mathbb Q$ (or, more generally, $x_0\in U$ as above) this can be equivalently stated as follows
 \[P^+_{i,\F,x_0}-P^-_{i,\F,x_0}=(x-x_0)^{k+1}Q(x)\quad\hbox{with $Q(x)\in\mathbb Q[x]$ of degree $\le g-k-1$ such that $Q(x_0)\ne 0$}\]
 (in particular, it follows that the index is at most $g-1$).The main result of this section is Proposition \ref{derivatives}, relating the index of a rational critical point with the dimension of the jump locus.

It is not difficult to exhibit cohomological rank functions with  critical points even of index zero. i.e. the function is non-differentiable at such a point. The  simple examples below serve also as illustration of the method of calculation provided by the results of Section 2.

 \begin{example}\label{counterexample}  Let $A=B\times E$, a principally polarized product of  a  principally polarized $(g-1)$-dimensional abelian variety $B$ and  an elliptic curve $E$. Let $\Theta_B$ be a principal polarization on $B$ and $p$ a closed point of $E$.  Let $\F=\mathcal{O}_B(\Theta_B)\boxtimes\mathcal{O}_E$. It is well known that:\\
 (a)  the FM transform -- on $E$ -- of  the sheaf $\OO_E$ is $k(\hat e)[-1]$, the one-dimensional skyscraper sheaf at the origin, in cohomological degree $1$.\\
 \noindent (b) The FM transform -- on B -- of the sheaf $\OO_B(-\Theta_B)$ is equal to $\OO_{\widehat B}(\Theta_{\widehat B})[-(g-1)]$.

\noindent By K\"unneth formula  it follows from (a) that $R^0\Phi_{\mathcal P}(\F)=0$, hence $h_{\F}^0(x\l)=0$ for $x<0$ (of course this was obvious from the beginning).
On the other hand, again from K\"unneth formula together with (a) and  (b) it follows that
 \[R\Phi_{\mathcal P^\vee}(\F^\vee)=R^g\Phi_{\mathcal P^\vee}(\F^\vee)[-g]=i_{\widehat B*}(\OO_{\widehat B}(\Theta_{\widehat B}))[-g],\]
  where $i_{\widehat B}:\widehat B\rightarrow \widehat A$ is the natural inclusion $\hat b\mapsto (\hat b,\hat e)$.
  Hence, for $x>0$
  \[h^0_{\F}(x\l)=(x)^gh^0_{R^g\Phi_{\mathcal P^\vee}(\F^\vee)}(\frac{1}{x})= x^g(1+\frac{1}{x})^{g-1}=x(1+x)^{g-1}\]
In conclusion
 \[h^0_{\F}(x\l)=\begin{cases}0&\hbox{for $x\leq 0$}\\
x(1+x)^{g-1}&\hbox{for $x\geq 0$}\end{cases}\]
(Of course the same calculation could have been worked out in a completely elementary way).
Hence $x_0=0$ is critical point of index zero.
\end{example}

\begin{example}\label{AJ-0} Let  $A$ be the Jacobian of a smooth curve of genus $g$, equipped with the natural principal polarization and let $i: C\hookrightarrow A$ be an Abel-Jacobi embedding. Let $p\in C$ and let $\F=i_*\OO_C((g-1)p)$. We claim that $x_0=0$ is a critical point of index zero for the function $h^0_\F(x\l)$.
Notice that $\F^{\vee}=i_*\omega_C(-(g-1)p)[1-g]$ and $\deg_C(\omega_C(-(g-1)p))=g-1$. Hence $R^0\Phi_{\mathcal P}(\F)=0$, while
\[R\Phi_{\mathcal P^\vee}(\F^\vee)=R^{g}\Phi_{\mathcal P^\vee}(\F^\vee)[-g]=R^1\Phi_{\mathcal P^\vee}\bigl(i_*\omega_C(-(g-1)p)\bigr)[-g]:=\mathcal H [-g]\]
 is a torsion sheaf in cohomological degree $g$ (supported at a translate of a theta-divisor, where it is of generic rank equal to $1$).
 From Proposition \ref{inversion-a} it follows that $h_{\F}^0(x\l)=0$ for $x\le 0$ and $h_{\F}^1(x\l)=0$ for $x\ge 0$. Hence, by (a) of \S1,
 \[h_{\F}^0(x\l)=\begin{cases}0&\hbox{for $x\leq 0$}\\
\chi_\F(x\l)=gx&\hbox{for $x\geq 0$}
\end{cases}\]
This proves what claimed.\\
One can show that $x_0=0$ is a critical point of index $g-d-1$ of the $h^0$-function of the sheaf $i_*\OO_C(dp)$, with $0\le d\le g-1$.
\end{example}

As it will be clear in the sequel, the previous examples are explained by the presence of a jump locus of codimension one.

\noindent\textbf{Jump loci. } We introduce some terminology. Let $(A,\l)$ be a polarized abelian variety. Let $\F\in\mathrm{D}^b(A)$ and $x_0\in\mathbb Q$. The \emph{jump locus} of the $i$-the cohomology of $\F$ at $x_0=\frac{a}{b}$ is the closed subscheme of $\widehat A$ consisting of the points $\alpha$ such that $h^i(A, (\mu_b^*\F)\otimes L^{ab}\otimes P_\alpha))$ is strictly greater than the generic value, where $L$ is a line bundle representing $\l$. A different choice of the line bundle $L$ changes the jump locus in a translate of it while a different fractional representation of $x_0$, say $x_0=\frac{ah}{bh}$ changes the jump locus in its  inverse image via the isogeny $ \mu_h:\widehat A\rightarrow \widehat A$. Therefore, strictly speaking,  for us the jump locus at $x_0$ of  a cohomological rank function $h^i_\F(x\l)$ will be an equivalence class of (reduced) subschemes with respect to the equivalence relation generated by translations
 and multiplication isogenies. In this paper we will be only concerned with the dimension of these loci. We will denote it by $\dim J^{i+}(\F\langle x_0\l\rangle)$.

 \begin{proposition}\label{derivatives} Let $\F\in\mathrm{D}^b(A)$. If $ x_0\in \mathbb Q$ is a critical point of index $k$ for $h^i_{\F,\l}$, then
  $\mathrm{codim}_{\widehat A}\, J^{i+}(\F\langle x_0\rangle)\le k+1$.
 \end{proposition}

\begin{proof}
We may assume that $ x_0=0$.
By Corollary \ref{inversion} we know that in a left neighborhood of $0$, $h_{\mathcal F}^i(x\l)=\frac{(-x)^g}{\chi(\l)} \chi_{\varphi_{\l}^*R^i\Phi_{\mathcal P}(\mathcal{F})}(-\frac{1}{x}\l)$ and in a right neighborhood of $0$, $h_{\mathcal F}^i(x\l)=\frac{x^g}{\chi(\l)} \chi_{\varphi_\l^*R^{g-i}\Phi_{\mathcal P^\vee}(\mathcal{F}^{\vee})}(\frac{1}{x}\l)$.
We denote
\[P_1(x):=\chi_{\varphi_\l^*R^i\Phi_{\mathcal P}(\mathcal{F})}(x\l)=a_gx^g+a_{g-1}x^{g-1}+\cdots+a_{1}x+a_0\]
 and
\[P_2(x):=\chi_{\varphi_\l^*R^{g-i}\Phi_{\mathcal P^\vee}(\mathcal{F}^{\vee})}(x\l)=b_gx^g+b_{g-1}x^{g-1}+\cdots+b_{1}x+b_0.\]
It follows $h_{\mathcal{F}}^i(x\l) $ is strictly of class $\mathcal C^k$ at $0$ if and only if
\begin{equation}\label{change-sign}(-1)^{j}a_{g-j}=b_{g-j}\quad\hbox{for \ $j=0,\>\cdots\>, k$}\quad \mathrm{and}  \quad (-1)^{k+1}a_{g-k-1}\neq b_{g-k-1}
\end{equation}

We also note that for a coherent sheaf $\mathcal Q$, $$\chi_{\mathcal Q}(x\l)=\int_{A}\mathrm{ch}(\mathcal Q)e^{x\l}=\sum_{j\geq 0}\frac{1}{(g-j)!}(\mathrm{ch}_j(\Q)\cdot \l^{g-j})_Ax^{g-j}.$$

On the other hand, by Grothendieck duality (\ref{mukai-3}) we have $\Phi_{\mathcal P}(\mathcal F)^\vee=\Phi_{\mathcal{P}^{\vee}}(\mathcal{F}^{\vee})[g]$. Thus we have a natural homomorphism $R^{g-i}\Phi_{\mathcal{P}^{\vee}}(\mathcal{F}^{\vee})\rightarrow \mathcal{H}om(R^i\Phi_{\mathcal P}(\mathcal{F}), \mathcal{O}_{\widehat A}):=\mathcal H$ by the Grothendieck spectral sequence. By base change and Serre duality such homomorphism is an isomorphism of vector bundles on the open set $V$ whose closed points are the $\alpha\in\widehat A$ such that  $h^i(\mathcal{F}\otimes P_{\alpha})$ takes the generic value.  Now assume that the complement of $V$, i.e.  a representative of $J^{i+}(\F)$, has codimension $>k+1$. Hence $\mathrm{ch}(R^{g-i}\Phi_{\mathcal{P}^{\vee}}(\mathcal{F}^{\vee}))-\mathrm{ch}(\mathcal H)\in \mathrm{CH}^{> k+1}(\widehat{A}).$ Since $R^i\Phi_{\mathcal P}(\mathcal F)$ is a vector bundle on $V$, thus $\mathrm{ch}_j(\mathcal{H})=(-1)^j\mathrm{ch}_j(R^i\Phi_{\mathcal P})(\mathcal F)$ for $j\leq k+1$. This implies that
$$(-1)^ja_{g-j}=\frac{(-1)^j}{(g-j)!}(\varphi_\l^*\mathrm{ch}_j\big(R^i\Phi_{\mathcal P})(\mathcal F)\cdot \l^{g-j}\big)_A=\frac{1}{(g-j)!}(\varphi_\l^*\mathrm{ch}_j\big(R^{g-i}\Phi_{\mathcal P^{\vee}})(\mathcal F^{\vee})\cdot \l^{g-j}\big)_A=b_{g-j},$$
for $j=0,\ldots, k+1$, which  contradicts  (\ref{change-sign}).  We conclude the proof.

 \end{proof}

 \section{Generic vanishing, M-regularity and IT(0) of  $\mathbb Q$-twisted sheaves on abelian varieties}

The notions of GV, M-regular and IT(0)-sheaves (and other related ones) are useful in the study of the geometry of abelian and irregular varieties via the Fourier-Mukai transform associated to the Poincar\'e line bundle. In this section we extend such notions
 to the $\mathbb Q$-twisted setting.  In doing that we  don't claim any originality, as this point of view was already taken, at least implicitly, in the work \cite{pp3}, Proof of Thm 4.1, and goes back to work of Hacon (\cite{hac}). It turns out that  the $\mathbb Q$-twisted formulation of Hacon's criterion for being GV, and related results, is simpler and more expressive even for usual (non-$\mathbb Q$-twisted) coherent sheaves or, more generally, objects of $\mathrm{D}^b(A)$. In the last part of the section we go back to cohomological rank functions. First we show how they can be used to provide a characterization of M-regularity and related notions. Finally  we show  the maximal critical points relates to the notion of $\mathbb Q$-twisted GV  sheaves.

 As for jump loci (see the previous Section), one can define the \emph{cohomological support locus} of the $i$-th cohomology of the $\mathbb Q$-twisted object of $\mathrm{D}^b(A)$, say $\F\langle x_0\l \rangle$, as the equivalence class (with respect to the equivalence relation generated by translations and inverse images by multiplication-isogenies) of the loci
 \[\{\alpha\in\widehat A\>|\> h^i(A, (\mu_b^*\F)\otimes L^{ab}\otimes P_\alpha))>0\}.\]
 If $h^i_\F(x_0\l)=0$ it coincides with the jump locus while it is simply $\widehat A$ if $h^i_\F(x\l)>0$. Its dimension is well-defined, and we will denote it $\dim V^i(\F\langle x_0\l\rangle)$.

 The $\mathbb Q$-twisted object $\F\langle x_0\l\rangle $ is said to be \emph{GV}   if $\text{codim}_{\widehat A}V^i(\F\langle x_0\l\rangle)\ge i$ for all $i>0$.  It is said to be \emph{a M-regular sheaf}   if $\text{codim}_{\widehat A}V^i(\F\langle x_0\l\rangle)> i$ for all $i>0$.
 It is said to \emph{satisfy the index theorem with index 0, IT(0)} for short, if $V^i(\F\langle x_0\l\rangle)$ is empty for all $i\ne 0$.

 If $\F\langle x_0\rangle $ is $\mathbb Q$-twisted coherent sheaf or, more generally, a $\mathbb Q$-twisted object of $\mathrm{D}^b(A)$ such that $V^i(\F\langle x_0\rangle)$ is empty for $i<0$, such conditions can be equivalently stated described as follows:

  \begin{theorem}\label{A} (a) $\F\langle x_0\l\rangle $ is GV if and only if, for one (hence for all) representation $x_0=\frac{a}{b}$
 \[\Phi_{\mathcal P^\vee}(\mu_b^*\F^\vee\otimes L^{-ab})=R^g\Phi_{\mathcal P^\vee}(\mu_b^*\F^\vee\otimes L^{-ab})[-g]\>.\>\footnote{This condition is usually expressed by saying that $\mu_b^*\F^\vee\otimes L^{-ab}$ satisfies the Weak Index Theorem with index $g$.}\]
 If this is the case
 \[R^i\Phi_{\mathcal P}(\mu_b^*(\F)\otimes L^{ab})=\mathcal{E}xt^i_{\OO_{\widehat A}}(R^g\Phi_{\mathcal P^\vee}(\mu_b^*\F^\vee\otimes L^{-ab}),\OO_{\widehat A})\]

\noindent (b) Assume that $\F\langle x_0\l\rangle $ is GV. Then it is M-regular if and only if the sheaf $R^g\Phi_{\mathcal P^\vee}(\mu_b^*\F^\vee\otimes L^{-ab})$ is torsion-free.

\noindent (c) Assume that $\F\langle x_0\l\rangle $ is GV, Then it is IT(0) if the sheaf $R^g\Phi_{\mathcal P^\vee}(\mu_b^*\F^\vee\otimes L^{-ab})$ is locally free. Equivalently
\begin{equation}\label{it0}\Phi_{\mathcal P}(\mu_b^*(\F)\otimes L^{ab})=R^0\Phi_{\mathcal P}(\mu_b^*(\F)\otimes L^{ab})\>.
\end{equation}
\end{theorem}

 These results follow immediately from the same statements for coherent sheaves or objects in $\mathrm{D}^b(A)$, see e.g. the survey \cite[\S1]{msri}, or \cite[\S3]{pp2}, where the subject is treated in much greater generality.

 In this language  well known criteria of Hacon (\cite{hac})  can be stated as follows:

  \begin{theorem}\label{B}  \noindent (a) $\F\langle x_0\l\rangle$ is GV if and only if $\F\langle (x_0+x)\l\rangle$ is IT(0) for sufficiently small $x\in \mathbb Q^+$. Equivalently
  $\F\langle x_0\l\rangle$ is GV if and only if $\F\langle (x_0+x)\l\rangle$ is IT(0) for all $x\in \mathbb Q^+$.

  \noindent (b) If $\F\langle x_0\l\rangle$ is GV but not IT(0) then $\F\langle (x_0-x)\l\rangle$ is not GV for all $x\in\mathbb Q^+$.

 \noindent (c)  $\F\langle x_0\l\rangle$ is IT(0) if and only if $\F\langle (x_0-x)\l\rangle $ is IT(0) for sufficiently small $x\in\mathbb Q^+$.
  \end{theorem}

  \begin{proof} (a) Let $x_0=\frac{a}{b}$. We have that $\F\langle x_0\l\rangle$ is GV if and only if $\mu_b^*(\F)\otimes L^{ab}$ is GV. Hacon's criterion (see \cite[Thm A]{pp2}) states that this is the case if and only if
  \begin{equation}\label{haconbis}H^i(\mu_b^*(\F)\otimes L^{ab}\otimes \Phi^{\widehat A\rightarrow A}_{\mathcal P}(N^{-k})[g])=0
  \end{equation}
 for all $i\ne 0$ and  for all sufficiently big $k\in \mathbb Z$, where $N$ is an ample line bundle on $\widehat A$.  Equivalently (up to taking a higher lower bound for $k$),
 \begin{equation}\label{hacon}
 \mu_b^*(\F)\otimes L^{ab}\otimes \Phi^{\widehat A\rightarrow A}_{\mathcal P}(N^{-k})[g]\quad\hbox{is}\quad IT(0)
 \end{equation}
 for sufficiently big $k$.
   We take as $N=L_\delta$ a line bundle representing the polarization $\l_\delta$ dual to $\l$ (\cite{birke-lange} \S14.4). By  Prop 14.4.1 of \emph{loc. cit.}  we have that
\begin{equation}\label{np-2}
\varphi_{\l}^*\l_\delta=d_1d_g\l
\end{equation}
  and
  \begin{equation}\label{np-3}
  \varphi_{\l_\delta}\circ\varphi_{\l}=\mu_{d_1d_g}
  \end{equation}
  where $(d_1,\dots ,d_g)$ is the type of $\l$.
 Combining with (\ref{mukai-4}) we get
 \[\mu_{d_1d_gk}^*\Phi^{\widehat A\rightarrow A}_{\mathcal P}(L_\delta^{k})=  \varphi_{\l}^*\varphi_{\l_\delta}^*\mu_k^*\Phi^{\widehat A\rightarrow A}_{\mathcal P}(L_\delta^{k})=(\varphi_\l^*(L_\delta^{-k}))^{\oplus k^g\chi(\l_\delta)}=
 (L^{-d_1d_gk})^{\oplus k^g\chi(\l_\delta)}\]
 Loosely speaking, we can think of the vector bundle $\Phi^{\widehat A\rightarrow A}_{\mathcal P}(L_\delta^{k})$ as representative of $(-\frac{1}{d_1d_gk}\l)^{\oplus k^g\chi(\l_\delta)}$. It follows, after a little calculation,  that (\ref{hacon}), hence the fact that $\F\langle x_0\l\rangle$ is GV, is equivalent to the fact that $\F\langle (x_0+\frac{1}{d_1d_gk})\l\rangle$ is IT(0) for sufficiently big $k$.  This is in turn equivalent to the fact that $\F\langle (x_0+x)\l\rangle$ is IT(0) for sufficiently small $x\in \mathbb Q^+$ because the tensor product of an IT(0) (or, more generally, GV) sheaf and a locally free IT(0) sheaf is IT(0) (\cite[Prop. 3.1]{pp3} ). This proves the first statement of (a).
 The second statement follows  again from \emph{loc.cit.}\\
  (b) follows from (a).\\
  (c) is proved as (a) using  a similar Hacon's criterion telling that (\ref{it0}) is equivalent to the fact that
$\mu_b^*(\F)\otimes L^{ab}\otimes \Phi^{\widehat A\rightarrow A}_{\mathcal P}(N^{k})$ is IT(0) for sufficiently big $k$.
  \end{proof}

Using the cohomological rank functions on the left neighborhood of a rational point, we have the following characterization of GV-sheaves and M-regular sheaves.
\begin{proposition}\label{gv&M}
(a) \ $\F\langle x_0\l\rangle$ is GV, if and only if $h_{\F}^i((x_0-x)\l)=O(x^i)$ for sufficiently small $x\in\mathbb Q^+$, for all $i\geq 1$.

\noindent (b)  \  $\F\langle x_0\l\rangle$ is M-regular, if and only if $h_{\F}^i((x_0-x)\l)=O(x^{i+1})$ for sufficiently small $x\in\mathbb Q^+$, for all $i\geq 1$.
\end{proposition}
\begin{proof}
We may suppose that $x_0=0$. Then $\F$ is GV (resp. M-regular) is equivalent to say that $\mathrm{codim}\; R^i\Phi_{\mathcal{P}}(\F)\geq i$ (resp. $\mathrm{codim}\;R^i\Phi_{\mathcal{P}}(\F)> i$)  for all $i\geq 1$ (see \cite{pp2} Lemma 3.6). Then we conclude by Corollary \ref{inversion}.
\end{proof}

It turns out that, more generally, the notion of \emph{gv}-index (\cite{duke} Def. 3.1) can be extended to the $\mathbb Q$-twisted setting and described via cohomological rank functions as in Proposition \ref{gv&M}. We leave this to the reader.

It is likely that a sort converse of Proposition \ref{derivatives} holds, namely the rational critical points arise only in presence of non-empty jump loci, although not necessarily for the same cohomological index. A partial result in this direction is the following

\begin{proposition}\label{maxthreshold} Let  $ x_0\in \mathbb Q$. If the $\mathbb Q$-twisted sheaf $\F\langle x_0\l\rangle$ is GV but not IT(0) then $x_0\in S_{\F,\l}$. In fact it is the maximal element of $S_{\F,\l}$.
\end{proposition}
However notice that, given a coherent sheaf $\F$, in general there is no reason to expect that there is an $x_0\in\mathbb Q$ such that the hypothesis of the Proposition holds. In other words, the maximal critical point might be irrational.
\begin{proof}
 As before, we may assume that $ x_0=0$ and we need to compare the coefficients of the two polynomials $P_1(x)=\chi_{\varphi_{\l}^*R^0\Phi_{\mathcal P}(\mathcal{F})}(x)$ and $P_2(x)=\chi_{\varphi_{\l}^*R^{g}\Phi_{{\mathcal P}^{\vee}}(\mathcal{F}^{\vee})}(x)$. By assumption, $\F$ is a GV sheaf, hence by Theorem \ref{A}(a)  $R\Phi_{\mathcal{P}^{\vee}}(\F^\vee)=R^g\Phi_{\mathcal{P}^{\vee}}(\F^{\vee})[-g]$ and $R^i\Phi_{\mathcal P}(\F)=\mathcal{E}xt^i(R^g\Phi_{\mathcal{P}^{\vee}}(\F^{\vee}), \mathcal{O}_{\widehat{A}})$. Moreover the condition that $\F$ is GV but not IT(0) implies that $R^g\Phi_{\mathcal{P}^{\vee}}(\F^{\vee})$ is not locally free.  Hence for some $i>0$, $R^i\Phi_{\mathcal P}(\F)$ is nonzero. Thus, $h_{\F}^i(x\l)$ is nonzero for $x$ in a left neighborhood of $0$ and obviously $h_{\F}^i(x\l)=0$ for $x$ positive. Hence $x_0\in S_{\F,\l}$.
\end{proof}

\begin{remark}\label{maximal-example}
Under the above assumption, it is in general not true that $x_0$ is a critical point of $h_{\F}^0(x\l)$ as shown by  the following example.
Let $(A, \underline \theta)$ be a principally polarized abelian variety and let $\Theta$ be a theta-divisor. Let $\F=\mathcal O_A\oplus \mathcal O_{\Theta}(\Theta)$. Then $\F$ is GV and not IT(0). It is easy to see that \[h_{\F}^0(x\l)=\begin{cases}(1+x)^g&\hbox{for $x\geq -1$}\\
0&\hbox{for $x<-1$}\end{cases}\] However notice that $h^{g-1}_{\F}(x\l)=h^g_{\F}(x\l)=(-x)^g$
for $-1\leq x\leq 0$.
\end{remark}

 \section{Some integral properties of the coefficients}
 In this section we point out some interesting integrality property of the polynomials involved in cohomological rank functions. We will use the results of \S2 throughout.

 \begin{lemma}Let $\F\in \mathrm{D}^b(A)$. Assume that $h^i_{\F}(x\l)=P(x)$ is a polynomial function for $x$ in an interval $U_1\subseteq \mathbb R$. Then all coefficients of $P(x)$ belong to $\frac{1}{g!}\mathbb Z$.
 \end{lemma}
\begin{proof}
We already know that $P(x)\in \mathbb{Q}[x]$ is a polynomial of degree at most $g$.
We may choose $p$ sufficiently large such that there exists   $q\in \mathbb Z$ such that the numbers $\frac{q}{p},\ldots, \frac{q+g}{p}$ and $\frac{q}{p+1},\ldots, \frac{q+g}{p+1}$ belong to $U_1$. By definition, $a_i:=P(\frac{q+i}{p})=h_{\F}^i(\frac{q+i}{p})=\frac{1}{p^{2g}}h_{gen}^i(\F\otimes L^{p(q+i)})\in \frac{1}{p^{2g}}\mathbb Z$. Then we know that
$$P(x)=\sum_{i=0}^g\frac{a_i}{\prod_{j\neq i}(\frac{i-j}{p})}\prod_{j\neq i}(x-\frac{q+j}{p})=\sum_{i=0}^g\frac{a_i}{\prod_{j\neq i}(i-j)}\prod_{j\neq i}(px-q-j).$$
Hence all coefficients of $P(x)$ belong to $\frac{1}{g!}\frac{1}{p^{2g}}\mathbb Z$. Apply the same argument to $P(\frac{q}{p+1}),\ldots, P(\frac{q+g}{p+1})$, we see that all  coefficients of $P(x)$ belong to $\frac{1}{g!}\frac{1}{(p+1)^{2g}}\mathbb Z$. Hence they belong to $\frac{1}{g!}\mathbb Z$.
\end{proof}
\begin{remark}By a slightly different argument, we can say something more. Let $\frac{a}{b}\in U_1$. Then by Corollary \ref{inversion}, we know that, for $x>0$ small enough,  \begin{eqnarray}Q(x):=P(\frac{a}{b}-x)&=&x^g\frac{1}{\chi(\l)}\chi_{\varphi_{\l}^*R^i\Phi_{\mathcal P}(\mu_b^*\F\otimes L^{ab})}(\frac{1}{b^2}x\l)\\\nonumber
&=& \frac{1}{\chi(\l)}\sum_{k=0}^g\big(ch_k(\varphi_{\l}^*R^i\Phi_{\mathcal P}(\mu_b^*\F\otimes L^{ab}))\cdot \l^{g-k}\big)_A\>\frac{1}{b^{2g-2k}}x^k.\end{eqnarray}
Hence the coefficient $b_k$ of $x^k$ of $Q(x)$ is $\frac{1}{\chi(\l)}\big(ch_k(\varphi_{\l}^*R^i\Phi_{\mathcal P}(\mu_b^*\F\otimes L^{ab}))\cdot \l^{g-k}\big)_A\>\frac{1}{b^{2g-2k}}$. Note that \begin{eqnarray*}\big(ch_k(\varphi_{\l}^*R^i\Phi_{\mathcal P}(\mu_b^*\F\otimes L^{ab}))\cdot \l^{g-k}\big)_A=\big(\varphi_{\l}^*ch_k(R^i\Phi_{\mathcal P}(\mu_b^*\F\otimes L^{ab}))\cdot \l^{g-k}\big)_A\\=\frac{\deg\varphi_{\l}}{(d_1d_g)^{g-k}}\big( ch_k(R^i\Phi_{\mathcal P}(\mu_b^*\F\otimes L^{ab}))\cdot \l_{\delta}^{g-k}\big)_{\widehat{A}},\end{eqnarray*} where $\l_\delta$ is the dual polarization (\cite{birke-lange} \S14.4) and the last equality holds because of (\ref{np-2}). Moreover, since
\begin{equation}
\chi(\l_\delta)=\frac{(d_1d_g)^g}{\chi(\l)}
\end{equation}
 we note that the class $[\l_{\delta}]^{g-k}$ belongs to $(g-k)!\frac{(d_1d_g)^{g-k}}{d_g\cdots d_{k
+1}}H^{2g-2k}(\widehat{A}, \mathbb Z)$. On the other hand, it is clear that $[ch_k(R^i\Phi_{\mathcal P}(\mu_b^*\F\otimes L^{ab}))]\in \frac{1}{k!}H^{2k}(\widehat{A}, \mathbb Z)$. Thus $b_k\in (d_1\cdots d_k)\frac{(g-k)!}{k!}\frac{1}{b^{2g-2k}}\mathbb Z$. From this computation, we see easily that the coefficient $a_k$ of $x^k$ in $P(x)$ belongs to $ (d_1\cdots d_k)\frac{(g-k)!}{k!}\mathbb Z$.
\end{remark}
We have the following strange corollary.
\begin{corollary} Let $\F\in \mathrm{D}^b(A)$. Then $b^g\mid h_{gen}^i(\mu_b^*\F\otimes L^a)$ for all $b$ such that $(b, g!)=1$ and $a\in \mathbb Z$.
\end{corollary}
\begin{proof}Since $h_{\F}^i(x\l)$ is a polynomial of degree at most $g$ whose coefficients belong to $\frac{1}{g!}\mathbb Z$, we have that $g!b^gh_{\F}^i(\frac{a}{b}\l)\in \mathbb Z$. As  $(b, g!)=1$, we conclude that $ b^g\mid h_{gen}^i(\mu_b^*\F\otimes L^a)$.
\end{proof}

\section{GV-subschemes of principally polarized abelian varieties}

Let $(A,\underline\theta)$ be a $g$-dimensional  principally polarized abelian variety.
A subscheme $X$ of $A$  is called a \emph{GV-subscheme} if its twisted ideal sheaf
$\I_X(\Theta)$ is GV. This technical definition is motivated by the fact that the subvarieties $\pm W_d$ of Jacobians and $\pm F$, the Fano surface of lines of intermediate jacobians of cubic threefolds,  are  the only known examples of (non-degenerate) GV-subschemes.
We summarize some basic  results on the subject in use in the sequel.  One considers the ``theta-dual" of $X$, namely the cohomological support locus
\[V(X):=V^0(\I_X(\Theta))=\{\alpha\in \widehat A\>|\> h^0(\I_X(\Theta)\otimes P_\alpha)>0\}\]
equipped with its natural scheme structure (\cite{minimal} \S4).   Let $X$ be a  geometrically non-degenerate GV-subscheme of pure dimension $d$. Then

\noindent (a) \cite[Theorem 2(1)]{s} \emph{$X$ and $V(X)$ are reduced and irreducible.}

\noindent (b) (\cite{minimal}) \emph{$V(X)$ is a geometrically non-degenerate GV-scheme of pure dimension $g-d-1$ \emph{(the maximal dimension)}. Moreover $V(V(X))=X$ and both $X$ and $V(X)$ are Cohen-Macaulay.}

\noindent  (c) (\emph{loc. cit.}) \ \emph{$\Phi_{\mathcal P}(\OO_X(\Theta))=\bigl(\I_{V(X)}(\Theta)\bigr)^\vee$. Equivalently, by Grothendieck duality \emph{(see (\ref{mukai-3}))}, $\Phi_{\mathcal P^\vee}(\omega_X(-\Theta))=\I_{V(X)}(\Theta)[-d]$. }

\noindent (d)  (\emph{loc. cit.}) \ \emph{$X$ has minimal class $[X]=\frac{\underline\theta^{g-d}}{(g-d)!}$. }

In \emph{loc. cit.} it is conjectured that the converse of (d) holds. According to the conjecture of Debarre, this would imply that that the only geometrically non-degenerate GV-subschemes are the subvarieties $\pm W_i$ and $\pm F$ as above. We refer to the Introduction for what is known in this direction.

\noindent\textbf{Generalities on GV-subschemes. } We start with some  general results on GV-subschemes, possibly of independent interest.  The first Proposition does not follow from Green-Lazarsfeld's Generic Vanishing Theorem because a GV-subscheme can be singular. It does follow,   via Lemma \ref{normality} below, from the Generic Vanishing Theorem of \cite{pareschi} which works for \emph{normal} Cohen-Macaulay subschemes of abelian varieties. However the following ad-hoc proof is much simpler.
\begin{proposition}\label{GL} Let $X$ be a non-degenerate reduced GV-subscheme. Then its dualizing sheaf $\omega_X$ is a GV-sheaf.
\end{proposition}
\begin{proof} By Hacon's criterion (\ref{haconbis}), it is enough to show that
\[H^i(\omega_X\otimes \Phi_{\mathcal P}(\OO_A(-k\Theta))[g])=0\quad\hbox{ for $k$ sufficiently big}\]
We have that $\Phi_{\mathcal P}(\OO_A(-k\Theta))[g]$ is a vector bundle (in degree $0$) which will be denoted, as usual, $\widehat{\OO_A(-k\Theta)}$.
We can write
\[H^i(\omega_X\otimes \widehat{\OO_A(-k\Theta)})=H^i(\omega_X(-\Theta)\otimes\widehat{\OO_A(-k\Theta)}(\Theta))\]
Applying the inverse of $\Phi_{\mathcal P}$, namely $\Phi_{{\mathcal P}^\vee[g]}$ (see (\ref{mukai-0})) to the last formula of (c) of this section we get that $\omega_X(-\Theta)=\Phi_{\mathcal P}(\I_{V(X)}(\Theta))[g-d]$. Therefore, by (\ref{exchange})
\[H^i((\omega_X(-\Theta))\otimes \widehat{\OO_A(-k\Theta)}(\Theta)))
= H^{i+g-d}\Bigl(\I_{V(X)}(\Theta)\otimes \Phi_{\mathcal P}\bigl(\widehat{\OO_A(-k\Theta)}(\Theta)\bigr)\Bigr)\]
Clearly $\widehat{\OO_A(-k\Theta)} $ is an IT(0) sheaf for $k\ge 1$ ((\ref{mukai-4})). Therefore $\Phi_{\mathcal P}\bigl(\widehat{\OO_A(-k\Theta)}(\Theta)\bigr)$ is a sheaf (locally free)
in cohomological degree $0$. Hence the above cohomology groups vanish for $i>0$ because $\dim V(X)=g-d-1$.
\end{proof}

\begin{proposition}  Let $X$ be a non-degenerate reduced GV-subscheme. Then
\[\Phi_{\mathcal P^\vee}(\omega_X)=\bigl(\Phi_{\mathcal P}(\mathcal I_{V(X)})\bigr)(-\Theta)[g-d]\]
\end{proposition}
\begin{proof}
\begin{eqnarray*}\Phi_{\mathcal P}(\omega_X)&=&
\Phi_{\mathcal P}\bigl(\omega_X(-\Theta)\otimes\mathcal{O}_A(\Theta)\bigr)\\
&\buildrel{(\ref{mukai-2})}\over =&\Phi_{\mathcal P}(\omega_X(-\Theta))*\Phi_{\mathcal P}(\OO_A(\Theta)[g]\\
&\buildrel{(c),(\ref{mukai-4})}\over =&
=(-1)^*\mathcal I_{V(X)}(\Theta)[-d]*\OO_A(-\Theta)[g]\\
&\buildrel{(\ref{mukai-5})}\over =& (\Phi_{\mathcal P}(\mathcal I_{-V(X)}))(-\Theta)[g-d].
\end{eqnarray*}
\end{proof}

\begin{corollary}\label{gv-strong} Let $X$ be a non-degenerate reduced GV-subscheme. Then
\begin{equation}\label{strong}R^i\Phi_{\mathcal P^\vee}(\omega_X)=0 \quad\hbox{for $i\ne 0,d$}\quad\hbox{and}\quad R^d\Phi_{\mathcal P^\vee}(\omega_X)=k(\hat e).\footnote{The last assertion was proved under very general assumptions in \cite[Prop. 6.1]{catalans}.}
\end{equation}
\end{corollary}
\begin{proof} This follows from the fact that also $V(X)$ is reduced Cohen-Macaulay ((b) of this section). Therefore, by Proposition \ref{GL},  combined by the duality characterization of Theorem \ref{haconbis}(a)
\[\Phi_{\mathcal P}(\OO_{V(X)})=R^{d-g-1}\Phi_{\mathcal P}(\OO_{V(X)})[-(g-d-1)]\] Hence it follows from  the standard
exact sequence $0\rightarrow \I_{V(X)}\rightarrow \OO_A\rightarrow \OO_{V(X)}\rightarrow 0$ that $R^i\Phi_{\mathcal P}(\I_{V(X)})=0$ for $i\ne g-d,g$ and $R^g\Phi_{\mathcal P}(\I_{V(X)})=k(\hat e)[-g]$.
Therefore the assertion follows from the previous Proposition.
\end{proof}

 Corollary \ref{gv-strong} is quite strong. Its implications hold in a quite general context but, for sake of brevity,  here we will stick to an ad-hoc treatment of the case of GV-subschemes.

\begin{lemma}\label{cup} Let $X$ be a subscheme of an abelian variety $A$. If $X$ satisfies (\ref{strong}) then  the natural maps $\Lambda^iH^1(\OO_A)\rightarrow H^i(\OO_X)$ are isomorphisms for all $i<d$ and injective for $i=d$.
\end{lemma}
\begin{proof} It the first place we note that (as it is well known) for all $\F\in \mathrm{D}^b(A)$ we have a natural isomorphism
\begin{equation}\label{ext} H^i(A,\F)\cong \text{Ext}^{g+i}(k(\hat e),\Phi_{\mathcal P}(\F))
\end{equation}
where $\hat e$ denotes the origin of $\widehat A$. Indeed, by the Fourier-Mukai equivalence,
\[H^i(A,\F)=\text{Hom}_{D(A)}(\OO_A,\F[i])\cong
 \text{Hom}_{D(\widehat A)}(k(\hat e)[-g],\Phi_{\mathcal P}(\F))\cong \text{Ext}^{g+i}(k(\hat e),\Phi_{\mathcal P}(\F)) \]
 Applying (\ref{ext}) to $\F=\omega_X$ we are reduced to compute the hypercohomology spectral sequence
 \[\text{Ext}^{g+i-k}(k(\hat e),R^k\Phi_{\mathcal P}(\omega_X))\Rightarrow \text{Ext}^{g+i}(k(\hat e),\Phi_{\mathcal P}(\omega_X))\cong H^i(\omega_X) \]
 We recall that $\text{Ext}^j_{\widehat A}(k(\hat e),k(\hat e))=\Lambda^jH^1(\OO_A)$. Condition (\ref{strong}) makes the above spectral sequence very easy. In fact we get the maps
 \begin{eqnarray*}H^i(\omega_X)\cong \text{Ext}^{g+i}(k(\hat e), \Phi_{\mathcal P}(\omega_X))&\rightarrow &\text{Ext}^{g+i-d}(k(\hat e), R^d\Phi_{\mathcal P}(\omega_X)[d])\\
 &=&\text{Ext}^{g+i-d}(k(\hat e), k(\hat e))\\
 &=&\Lambda^{g+i-d}H^1(\OO_A)\\
 &=&\Lambda^{d-i}H^1(\OO_A)^\vee
 \end{eqnarray*}
which are isomorphisms for $i>0$ and surjective for $i=0$. The Lemma follows by duality. \end{proof}

\noindent\textbf{The Poincar\'e polynomial of a GV-subscheme. } As an  application of cohomological rank functions we prove that  the Hilbert polynomial of  geometrically non-degenerate GV-subschemes is the conjectured one.

\begin{theorem}\label{gv1} Let $X$ be a geometrically non-degenerate GV-subscheme of dimension $d$ of a principally polarized abelian variety $(A,
\underline\theta)$. Then
\[\chi_{\OO_X}(x\underline\theta)=\sum_{i=0}^d{g\choose i}(x-1)^i\]
\end{theorem}
 \begin{proof} We compute the functions $h^i_{\OO_X}(x\underline\theta)$ in a neighborhood of $x_0=1$. First we compute it in an interval $(1-\epsilon^-,1]$ as in Corollary \ref{inversion}. By (b) and (c) of this section we have that $R^0\Phi_{\mathcal P}(\OO_X(\Theta))=\OO_A(-\Theta)$, $R^d\Phi_{\mathcal P}(\OO_X(\Theta))=\omega_{V(X)}(-\Theta)$ and $R^i\Phi_{\mathcal P}(\OO_X(\Theta))=0$ for $i\ne 0,d$ (see e.g. \cite[Prop. 5.1(b)]{minimal}). Therefore, in a small interval $[-\epsilon^-,0]$
 \[h^i_{\OO_X(\Theta)}(y\underline\theta)=\begin{cases}(-y)^g\chi_{\OO_A(-\Theta)}(-\frac{1}{y})=y^g(1+\frac{1}{y})^g=(1+y)^g&\text{for }\>i=0\\
 (-y)^gQ(-\frac{1}{y})&\text{for } \> i=d\\
 0&\text{for } \> i\ne 0,d
 \end{cases}\]
 where $Q$ is the Hilbert polynomial of the sheaf $\omega_{V(X)}(-\Theta)$, hence a polynomial of degree $g-d-1$. It follows that  $(-y)^gQ(-\frac{1}{y})=y^{d+1}T(y)$, where $T$ is a polynomial of degree $g-d-1$ such that $T(0)\ne 0$. Setting $x=1+y$ we get that in the interval $(1-\epsilon^-,1]$ \footnote{actually in  Proposition \ref{gv2} below it will be shown that $\epsilon^-=1$, but this is not necessary for the present Theorem}
  \[h^i_{\OO_X}(x\underline\theta)=\begin{cases}x^g&\text{for }\>i=0\\
 (x-1)^{d+1}T(x-1)&\text{for } \> i=d, \>\text {where } \> \deg T=g-d-1 \> \> \text{and } \> T(1)\ne 0 \>\> \\
 0&\text{for } \> i\ne 0,d
 \end{cases}\]
 Writing $x^g$ as its Taylor expansion centered at $x_0=1$ this yields the equality of polynomials
 \[\chi_{\OO_X}(x\underline\theta)= \sum_{i=0}^g{g\choose i}(x-1)^i +(-1)^{d}(x-1)^{d+1}T(x-1).\]
 It follows that $\chi_{\OO_X}(x\underline\theta)$, which is a polynomial of degree $d$, is the Taylor expansion of $x^g$ at order $d$. \end{proof}

In the following proposition we compute  the cohomological rank functions (with respect to the polarization $\underline\theta$) of the structure sheaf of a GV-scheme. In particular, this answers to Question 8.10 of \cite{bps}, asking, in the present terminology, for the cohomological rank functions of the structure sheaf of a curve in its Jacobian.

 \begin{proposition}\label{gv2} In the same hypotheses of the previous Theorem
 \[h^0_{\OO_X}(x\underline\theta)=\begin{cases}0&\text{for }\> x\le 0\\
 x^g&\text{for }\> x\in[0,1]\\
 \chi_{\OO_X}(x\underline\theta)=\sum_{i=0}^d{g\choose i}(x-1)^i&\text{for } \> x\ge 1\\
 \end{cases}\]
 \end{proposition}
\begin{proof} The assertion for $x\le 0$ is obvious. The assertion for $x\ge 1$ follows from the fact that $\OO_X(\Theta)$ is a GV-sheaf (in fact M-regular), this last assertion being well known, as it follows at once from the definition of GV-subscheme and the exact sequence
\[0\rightarrow \mathcal{I}_X(\Theta)\rightarrow \OO_A(\Theta)\rightarrow \OO_X(\Theta)\rightarrow 0.\]
 Therefore, by Theorem \ref{B}(a), $\OO_X\langle (1+x)\underline\theta
\rangle$ is IT(0) for $x>0$. In the proof of the previous Theorem, we computed the function $h^0_{\OO_X}(x\underline\theta)=x^g$ for $x$ in an interval $(1-\epsilon^-,1]$. Therefore, to conclude the proof, we need to show that we can take $\epsilon^-=1$. By Proposition \ref{GL}, the dualizing sheaf $\omega_{V(X)}$ is a GV-sheaf. Hence, again by Theorem \ref{B}(a), $\omega_{V(X)}\langle x\underline\theta\rangle$ is IT(0) for $x>0$. Therefore both $R^0\Phi_{\mathcal P}(\OO_X(\Theta))\langle n\underline\theta\rangle=\OO_A(-\Theta)\langle n\underline\theta\rangle$ and $R^d\Phi_{\mathcal P}(\OO_X(\Theta))\langle n\underline\theta\rangle=\omega_{V(X)}(-\Theta)\langle n\underline\theta\rangle$ are IT(0) for $n>1$. Therefore one can take $\epsilon^-=1$ (see Remark \ref{eff-serre}).
\end{proof}

A consequence of Corollary \ref{gv-strong} and Theorem \ref{gv1} we get

\begin{theorem}\label{hodgenumber} Let $X$ be a geometrically non-degenerate $d$-dimensional $GV$-subscheme of a $g$-dimensional p.p.a.v. $A$. Then
\[h^i(\OO_X)={g\choose i}\quad\hbox{for all $i\le d$}\]
\end{theorem}
\begin{proof} Lemma \ref{cup} implies that the natural map $\Lambda^iH^1(\OO_A)\rightarrow H^i(\OO_X)$ is an isomorphism for $i<d$ and injective for $i=d$. By Theorem \ref{gv1} the coefficient of degree $0$ of the Poincar\'e polynomial, namely $\chi(\OO_X)$, is equal to $\sum_{i=0}^d(-1)^i{g\choose i}$. The result follows.
\end{proof}

If one believes to  the conjectures mentioned at the beginning of this section, geometrically non-degenerate GV-subschemes should be normal with rational singularities. On a somewhat different note, we take the opportunity to prove some partial results in this direction, using the results of \cite{s} and \cite{cps}.
\begin{lemma}\label{normality}
Let $(A, \underline\theta)$ be an indecomposable PPAV. Assume that  $X$ is a geometrically non-degenerate GV-subscheme. Then $X$ is normal.
\end{lemma}

\begin{proof}
We know that $X$ is reduced and irreducible by (a) of this section.  Since an irreducible theta divisor is smooth in codimension $1$, by the argument in \cite[Theorem 4.1]{cps},  GV-subschemes of $A$ are smooth in codimension $1$.  Since $X$ is Cohen-Macaulay, we conclude that $X$ is normal by Serre's criterion.

\end{proof}
\begin{corollary}\label{rational-singularity}
Let $(A, \underline\theta)$ be an indecomposable PPAV. Assume that $X$ is a geometrically non-degenerate GV-subscheme of dimension $2$. Then $X$ has rational singularities.
\end{corollary}
\begin{proof}
Fix $\mu: X'\rightarrow X$ a resolution of singularities. By Lemma \ref{normality}, we only need to prove that $\mu_*\omega_{X'}=\omega_X$ to conclude that $X$ has rational singularities.

We first claim that $h^1(\mathcal O_{X'})=h^0(\Omega_{X'}^1)=g$. Assume the contrary. Then the Albanese variety $ A_{X'}$ has dimension $h^1(\mathcal O_{X'})>g$. We consider the commutative diagram
\begin{eqnarray*}
\xymatrix{
X'\ar[dr]_{\tau}\ar[d]^{\mu}\ar[r]^{a_{X'}} & A_{X'}\ar[d]\\
X \ar@{^{(}->}[r]& A.}
\end{eqnarray*}
Since $X$ is normal, $\mu_*\mathcal{O}_{X'}=\mathcal{O}_X$. Hence $h^1(\mathcal O_{X'})=h^1(\mathcal{O}_X)+h^0(X, R^1\mu_*\mathcal{O}_{X'})>g$. Thus  $h^0(X, R^1\mu_*\mathcal{O}_{X'})>0$ and hence there exists an irreducible curve $C$ on $X'$ which is contracted by $\mu$ and $a_{X'}\mid_C$ is generically finite onto its image. In particular, there exists a holomorphic $1$-form $\omega_0\in H^0(X', \Omega_{X'}^1)$ such that $(\omega_0)_{|C}$ is non-zero. Pick $p\in C$  a general point
 and consider the following local calculation around $p$. Let $x, y$  be local analytic coordinates of $X'$ around $p$ and assume that $C$ is defined by $y=0$. We may assume that $\tau(p)$ is the origin of $A$ and $\tau(x, y)=(f_1(x, y), \ldots, f_g(x,y))$ in an analytic neighborhood of $p$. Let $m$ be the multiplicity of $C$ in the fiber $\mu^*(0)$. Then the holomorphic functions $f_i$ can be written as  $y^mg_i(x,y)$ with $g_i$ holomorphic around $p$. For each holomorphic $1$-form $\omega\in H^0(A, \Omega_A^1)$,  we write $\tau^*\omega=h_1dx+h_2dy$ in a neighborhood of $p$. Then $y^m\mid h_1$ and $y^{m-1}\mid h_2$. Thus for any $s\in \tau^*H^0(A, \Omega_A^2)\subset H^0(X', K_{X'}) $, the corresponding divisor $D(s)$ has multiplicity of $C$ $\geq  2m-1$. On the other hand, since $\omega_0\mid_C\neq 0$, writing locally $\omega_0=g_1dx+g_2dy$ around $p$,  we have that $(g_1)_{|_C}$ is non-zero. Hence there exists $t\in \omega_0\wedge \tau^*H^0(A, \Omega_A^1)\subset H^0(X', K_{X'})$  such that the corresponding divisor $D(t)$ whose multiplicity of $C$  is $m-1$. But then have a contradiction, since, by a result of  Schreieder mentioned in the Introduction, the natural map  $H^0(X', K_{X'})\simeq \tau^*H^0(A, \Omega_A^2)$ is an isomorphism (\cite[Theorem 2]{s}). This proves what claimed. It follows that the natural map $H^1(\OO_A)\rightarrow H^1(\OO_{X^\prime})$ is an isomorphism.

 Consider the short exact sequence $$0\rightarrow \mu_*\omega_{X'}\rightarrow \omega_X\rightarrow \tau\rightarrow 0.$$ By Theorem \ref{hodgenumber} and the result of Schreieder it follows  that $H^0(\mu_*\omega_{X'})\rightarrow H^0(\omega_X)$ is an isomorphism. By Theorem \ref{hodgenumber} and the above claim it follows that the map $H^1(\mu_*\omega_{X'})\rightarrow H^1(\omega_X)$ is an isomorphism. Thus $h^0(\tau)=0$ hence the cohomological support locus $V^0(\tau)$ is strictly contained in $\widehat A$. On the other hand, we know that $\omega_X$ is M-regular by Corollary \ref{gv-strong}. This, together with the fact that $\mu_*\omega_{X'}$ is GV (\cite{hac}) yields that  $\tau$  is M-regular. Since $V^0(\tau)$ is strictly contained in $\widehat A$, this implies   (as it is well known,  see e.g. \cite[Lemma 1.12(b)]{msri}) that $\tau=0$.
\end{proof}

\begin{corollary}\label{l-t} In the hypothesis of the previous Corollary, let $X^\prime$ be any desingularization of $X$. Then the induced morphism $\tau:X^\prime \rightarrow A$ is the Albanese morphism of $X^\prime$.
\end{corollary}
\begin{proof} In view of the fact that $h^1(\OO_{X^\prime})=g=\dim A$, it is enough to prove that $\tau$ does not factor through any non-trivial isogeny. This means  that if $\alpha\in \widehat A$ is such that  $\tau^*P_\alpha$ is trivial
 then $\alpha=\hat e$. But this follows from the last part of Lemma \ref{gv-strong}, which implies by base change that the cohomological support locus $V^d(\omega_X)=\{\hat e\}$ and  Lemma \ref{rational-singularity}.
Alternatively, one can use the results of \cite{lt}.
\end{proof}

\section{Cohomological rank functions of the ideal of one point, multiplication maps of global sections, and normal generation of abelian varieties}

We refer to the Introduction for a general presentation of the contents of this section. Let $A$ be an abelian variety and $\l$ and $\n$ be polarizations on $A$ (in our applications $\n$ will be a multiple of $\l$). Assume moreover that $\n$ is basepoint free. Let $N$ be an ample and basepoint free line bundle representing $\n$. We consider the evaluation bundle of $N$, defined by the exact sequence
 \begin{equation}\label{evaluation}0\rightarrow M_N\rightarrow H^0(N)\otimes\OO_A\rightarrow N\rightarrow 0
 \end{equation}
 Finally, let $p\in A$. We consider the  cohomological rank functions $h^i_{\I_p}(x\l)$ and $h^i_{M_N}(x\n)$.\footnote{Note that they  don't depend respectively on $p$ and on $N$.} In both cases for $x\ge 0$  the functions are zero for $i\ge 2$ and for $x\le 0$ they are  zero for $i\ne 1, g$. We consider their maximal critical points, namely
 \begin{eqnarray*}\beta({\l})
 &=&\inf \{x\in\mathbb Q\>|\> h^1_{\I_p}(x\l)=0\}\\
s({\underline n})&=&
\inf\{x\in\mathbb Q\>|\> h^1_{M_N}(x\underline n)=0\}
\end{eqnarray*}
As it is easy to see, the problem illustrated by Remark \ref{maximal-example} about Proposition \ref{maxthreshold} does not occur for these two sheaves. Hence if $x_0\in\mathbb Q$ is such that $\I_p\langle x_0\l\rangle$ (resp. $M_N\langle x_0\n\rangle$) is GV but not IT(0) then $\beta(\l)$ (resp. $s(\n)$) is a critical point of $h^i_{\I_p}(x\l)$ (resp. $h^i_{M_N}(x\n)$) for $i=0,1$, in fact the maximal one. In any case, for $x\in \mathbb Q$, the fact that $x>\beta(\l)$ (resp. $y> s(\n)$) is equivalent to the fact that $\I_p\langle x\l\rangle$ (resp. $M_N\langle y\n\rangle$) is IT(0).

Let us spell what the IT(0) (resp. GV) condition mean for the above $\mathbb Q$-twisted sheaves. For $x=\frac{a}{b}\in \mathbb Q^{>0}$, the fact that $\I_p\langle x\l\rangle$ is IT(0) means that
\[h^1(\mu_b^*(\I_p)\otimes L_\alpha^{ab})=0\]
for all line $\alpha\in \widehat A$, where as usual $L_\alpha^{ab}$ denotes $L^{ab}\otimes P_\alpha$. This means that the finite scheme $p+\mu_b^{-1}(0)$  imposes independent conditions to the line bundles $L^{ab}_\alpha$ for all $\alpha\in\widehat A$. Since the $L^{ab}_\alpha$'s are all translates of the same line bundle this means that  \emph{for all} $p\in A$ the finite scheme $p+\mu_b^{-1}(0)$  impose independent conditions to the global sections of the line bundle $L^{ab}$ (hence the same happens for \emph{all} line bundles $L_\alpha^{ab}$). This condition can be interpreted as basepoint-freeness for the fractional polarization $x\l$. Note that if $x\in \mathbb Z$, writing $x=\frac{xb}{b}$ one finds back the usual basepoint-freeness. In turn the fact that $\I_p\langle x\l\rangle$ is GV but not IT(0) means that for all $\alpha$ in a proper closed subset of $ \widehat A$  the finite scheme $p+\mu_b^{-1}(0)$  does not impose independent conditions to the global sections of the line bundle $L^{ab}_\alpha$. As above this means that for  $p$ in a proper subset of $ A$ the finite scheme $p+\mu_b^{-1}(0)$ does not impose independent conditions to the global sections of $L^{ab}$  (hence the same property holds for all line bundles $L^{ab}_\alpha$). Again, for $x\in \mathbb Z$ one finds back the usual notion of base points and base locus. It follows that $\I_p(L)$ is any case GV and it is IT(0) if and only if $\l$ is basepoint free. In other words: $\beta(\l)\le 1$ and equality holds if and only if $\l$ has base points.

Similarly, for $y=\frac{a}{b}$, the fact that $M_N\langle y\n\rangle$ is IT(0) (resp. GV) means that
\[h^1(\mu_b^*(M_N)\otimes N^{ab}_\alpha)=0\]
for all (resp. for general) $\alpha\in \widehat A$. Pulling back the exact sequence (\ref{evaluation}) via $\mu_b$ and tensoring with $L^{ab}_\alpha$ this has the meaning mentioned in the introduction, namely that the multiplication maps obtained by composing with the natural inclusion $H^0(N)\hookrightarrow H^0(\mu_b^*N)$
 \begin{equation}\label{mult-frac-new}H^0(N)\otimes H^0(N^{ab}_\alpha)\rightarrow H^0(\mu_b^*(N)\otimes N^{ab}_\alpha)
\end{equation}  are surjective for all (resp. for general) $\alpha\in \widehat A$. The above maps (\ref{mult-frac-new}) can be thought as the multiplication maps of global sections of $N$ and of a representative of   the rational power  $N^{\frac{a}{b}}$ (twisted by $P_\alpha$)

\begin{proposition} Let $(A,\n)$ be a polarized abelian variety and assume that $\underline n$ is basepoint free. Let $p\in A$. For $i=0,1$ and $y<1 $
\[h^i_{\I_p}(y\underline n)=\frac{(1-y)^g}{\chi(\n)}h^i_{M_N}((-1+\frac{1}{1-y})\underline n)\]
Consequently
\begin{equation}\label{s-r}s(\n)=-1+\frac{1}{1-\beta(\n)}=\frac{\beta(\n)}{1-\beta(\n)}
\end{equation}
\end{proposition}
\begin{proof}
We can assume that $p=e$ (the origin of $A$). The essential point of the proof is that
\begin{equation}\label{f}\varphi_{\underline n}^*(R^0\Phi_{\mathcal P}(\I_e(N)))=M_N\otimes N^{-1}
\end{equation}
Indeed, by the exact sequence $0\rightarrow \I_e(N)\rightarrow N\rightarrow N\otimes k(e)\rightarrow 0$ it follows that $R^0\Phi_{\mathcal P}(\I_e(N))$ is the kernel of the map
\[R^0\Phi_{\mathcal P}(N)=\widehat N\buildrel f \over \rightarrow R^0\Phi_{\mathcal P}(N\otimes k(e))=\OO_{\widehat A}.\]
 By (\ref{mukai-4}) the map $\varphi_{\underline n}^*(f)$ is identified to a  map
$H^0(N)\otimes N^{-1}\rightarrow \OO_A$ which is easily seen to be the evaluation map tensored with $N^{-1}$.

Next, we notice that, since the polarization $\n$ is assumed to be basepoint free, we have that
\begin{equation}\label{R0}
\Phi_{\mathcal P}(\I_e(N))=R^0\Phi_{\mathcal P}(\I_e(N))
\end{equation}
 To prove this, we first notice that  $H^i(\I_e\otimes N_\alpha)=0$ for all $\alpha\in \widehat A$ and $i>1$. By base change this implies that the support $R^i\Phi_{\mathcal P}(\I_e(N))$ is equal to $V^1(\I_e(L))=\{\alpha\in\widehat A\>|\> h^1(\I_e\otimes N\otimes P_\alpha)>0\}$ which is non-empty if and only if $\n$ has base points. This proves (\ref{R0}).

  Therefore, by  Proposition \ref{inversion-a} and degeneration of the spectral sequence computing the hypercohomology, we have that for $i=0,1$ and  $t< 0$
\[h^i_{\I_e(N)}(t\n)=\frac{(-t)^g}{\chi(\n)}h^i_{\varphi_{\underline n}^*R^0\Phi_{\mathcal P}(\I_0(N))}(-\frac{1}{t}\underline n)\buildrel{(\ref{f})}\over =\frac{(-t)^g}{\chi(\n)}h^i_{M_N}((-1-\frac{1}{t})\underline n)\]
The first statement of the Proposition follows setting $y=1+t$. The second statement follows from the first one.
\end{proof}

Applying the previous proposition to divisible polarizations $\n=h\l$ we get Theorem \ref{b-s} of the Introduction, namely
\begin{corollary} Let $(A,\l)$ be a polarized abelian variety and let $h$ be an integer such that $h\l$ is basepoint free \emph{(hence $h\ge 2$, and $h\ge 1$ if $\l$ is basepoint-free)}. Then
\begin{equation}\label{s-beta}
s(h \l)= \frac{\beta(\l)}{h-\beta(\l)}
\end{equation}
Consequently:

\noindent (a)
 \[s(h \l)\le \frac{1}{h-1}\]
and equality  holds if and only if $\l$ has base points.

\noindent (b) Assume that $\l$ is base point free. Then $s(\l)< 1$ if and only if $\beta(\l)<\frac{1}{2}$. In particular, if $\beta(\l)<\frac{1}{2}$ then $\l$ is normally generated.
\end{corollary}
\begin{proof} (a) By definition, $h^i_{\I_e}(x(h\l))=h^i_{\I_e}((xh)\l)$. Hence $\beta(h\l)=\frac{1}{h}\beta(\l)$.
By the previous Proposition,
\[s(h\l)
=\frac{\beta(h\l)}{1-\beta(h\l)}=\frac{\beta(\l)}{h-\beta(\l)}\]
 The last statement follows from the fact that $\beta(\l)\le 1$ and equality holds if and only if $\l$ has base points.

\noindent (b) The first assertion follows immediately from  (\ref{s-beta}). Concerning the last assertion, we have that $s(\l)<1$ if and only if the multiplication maps
\[H^0(L)\otimes H^0(L_\alpha)\rightarrow H^0(L^2_\alpha)\]
are surjective for all $\alpha\in\widehat A$. A well known argument (e.g. \cite{kempf}, proof of Thm 6.8(c) and Cor. 6.9, or \cite{birke-lange}, proof of Theorem 7.3.1) proves that  this implies that $L_\alpha$ is normally generated for all $\alpha\in\widehat A$.
\end{proof}

Item (b) as well as the case $h=2$ of  item (a) of the above Corollary have been already commented in the Introduction. Here we note that the Proposition implies that, for an integer $h\ge 2$ and $\frac{a}{b}\ge \frac{1}{h-1}$ the ``fractional" multiplication maps of global sections
\[H^0(L^h)\otimes H^0(L^{hab}_\alpha)\rightarrow H^0(\mu_b^*(L^h)\otimes L^{hab}_\alpha)\]
are surjective for general $\alpha\in \widehat A$ and in fact for all $\alpha \in \widehat A$  as soon as $\frac{a}{b}> \frac{1}{h-1}$ or $\l$ is basepoint free. This is much stronger than the known results on the subject. For example, for $h=3$ Koizumi's theorem on projective normality, in a slightly stronger version (\cite{kempf} Cor.6.9, \cite{birke-lange} Th. 7.3.1) tells that the above maps are surjective for all $\alpha\in\widehat A$ for $b=1$ and $a=\frac{2}{3}$
while the Corollary asserts that the same happens for $\frac{a}{b}>\frac{1}{2}$. Moreover, for the critical value $\frac{a}{b}=\frac{1}{2}$, the Corollary tells that the maps
\[H^0(L^3)\otimes H^0(L^{6}_\alpha)\rightarrow H^0(\mu_2^*(L^3)\otimes L^{6}_\alpha)\]
are surjective for general $\alpha\in \widehat A$  and in fact for all $\alpha\in\widehat A$ as soon as $\l$ is basepoint free. For arbitrary $h$ the same happens for the ``critical" maps
 \[H^0(L^h)\otimes H^0(L^{h(h-1)}_\alpha)\rightarrow H^0(\mu_{h-1}^*(L^h)\otimes L^{h(h-1)}_\alpha)\]
Note that, when $\l$ is a principal polarization, the dimension of the source of the above maps  is equal to the dimension of the target, namely $(h^2(h-1))^g$.

\providecommand{\bysame}{\leavevmode\hbox
to3em{\hrulefill}\thinspace}

\end{document}